\documentclass[11pt]{elsarticle} 
\usepackage{a4,amssymb, amsmath, amsthm}
\usepackage{graphics,color}
\usepackage{graphicx}
\usepackage{epsfig}
\usepackage{subfig}

\newtheorem{example}{Example}

\newtheorem{thm}{Theorem}[section]

\newtheorem{lem}[thm]{Lemma}

\newtheorem{problem}[thm]{Problem}
\newtheorem{pblm}{Problem}

\newtheorem{pr}{Problem}

\newtheorem{alg}{Algorithm}

\theoremstyle{definition}

\theoremstyle{remark}
\newtheorem{rem}[thm]{Remark}

\newcommand*\diff{\mathop{}\!\mathrm{d}}

\DeclareMathOperator{\dt}{\mathrm{d}t}

\DeclareMathOperator{\pt}{\partial_t}

\begin{document}

\title{Space-time adaptive finite elements for nonlocal parabolic variational inequalities}
%\author{ Heiko Gimperlein\thanks{Maxwell Institute for Mathematical Sciences and Department of Mathematics, Heriot--Watt University, Edinburgh, EH14 4AS, United Kingdom, email: h.gimperlein@hw.ac.uk.}  \thanks{Institute for Mathematics, University of Paderborn, Warburger Str.~100, 33098 Paderborn, Germany.\newline H.~G.~acknowledges support by ERC Advanced Grant HARG 268105 and the EPSRC Impact Acceleration Account. J.~S.~was supported by The Maxwell Institute Graduate School in Analysis and its
%Applications, a Centre for Doctoral Training funded by the UK Engineering and Physical
%Sciences Research Council (grant EP/L016508/01), the Scottish Funding Council, Heriot-Watt
%University and the University of Edinburgh.
%} \and  Jakub Stocek${}^\ast$ }\date{\today}

\author[hw,pb]{Heiko~Gimperlein\corref{cor1}\fnref{fn1}}
\ead{h.gimperlein@hw.ac.uk}
\author[hw]{Jakub~Stocek\fnref{fn2}}
\ead{js325@hw.ac.uk}

\cortext[cor1]{Corresponding author}
\fntext[fn1]{H.~G.~acknowledges support by ERC Advanced Grant HARG 268105.}
\fntext[fn2]{J.~S.~was supported by The Maxwell Institute Graduate School in Analysis and its
Applications, a Centre for Doctoral Training funded by the UK Engineering and Physical
Sciences Research Council (grant EP/L016508/01), the Scottish Funding Council, Heriot-Watt
University and the University of Edinburgh.}
\address[hw]{Maxwell Institute for Mathematical Sciences and Department of Mathematics, Heriot--Watt University, Edinburgh, EH14 4AS, United Kingdom}
\address[pb]{Institute for Mathematics, University of Paderborn, Warburger Str.~100, 33098 Paderborn, Germany}

\begin{abstract}
This article considers the error analysis of finite element discretizations and adaptive mesh refinement procedures for nonlocal dynamic contact and friction, both in the domain and on the boundary. For a large class of parabolic variational inequalities associated to the fractional Laplacian we obtain a priori and a posteriori error estimates and study the resulting space-time adaptive mesh-refinement procedures. Particular emphasis is placed on mixed formulations,  which include the contact forces as a Lagrange multiplier. Corresponding results are presented for elliptic problems. Our numerical experiments for $2$-dimensional model problems confirm the theoretical results: They indicate the efficiency of the a posteriori error estimates and illustrate the convergence properties of space-time adaptive, as well as uniform and graded discretizations.
\end{abstract}
\begin{keyword}
fractional Laplacian; variational inequality; space-time adaptivity; a posteriori error estimates; a priori error estimates; dynamic contact.
\end{keyword}
\maketitle

\vskip 1.0cm

\section{Introduction}\label{intro}
Variational inequalities for time-dependent nonlocal differential equations have attracted recent interest in a wide variety of applications. Classically, parabolic obstacle problems arise in the pricing of American options with jump processes \cite{ContTankov,Wilmott}; current advances include their regularity theory \cite{BarriosFigalliRO2018,Silvestre} and the a priori analysis of numerical approximations \cite{Salgado,GuanGunzburger}. Mechanical problems naturally involve contact and friction at the boundary with surrounding materials. For nonlocal material laws, they are intensely studied in peridynamics  \cite{Gunzburger2018vi,Madenci,Du3}, but even for local material laws boundary integral formulations give rise to nonlocal problems \cite{eck,HeikoContact,gwinsteph}. Friction also plays a role in nonlocal evolution equations in image processing \cite{Image2,Image3,Image1}, and obstacle problems arise in the study of nonlocal interaction energies in kinetic equations \cite{cdm}.

For local differential equations, the pure and numerical analysis of variational inequalities has a long history \cite{petrosyan2012regularity}, especially motivated by contact problems in mechanics \cite{hlava,oden,wrig}. Of particular current interest in numerical analysis have been dynamic contact problems for time-dependent equations, including adaptive mesh refinements \cite{HAGER2012,wohlmuth}, high-order \cite{BanzStephan2014} and Nitsche methods \cite{burman2018nitsche,Chouly2017}. Their analysis is crucial for applications from tire dynamics \cite{Banz2016,HAGER2012} to blood flow in aortic valves \cite{Gerbeau2009}.\\

This article considers the systematic error analysis of the four standard parabolic variational inequalities \cite{petrosyan2012regularity} associated to the fractional Laplacian as a nonlocal model operator: obstacle problems and friction, both in the domain and in the boundary. Particular emphasis is placed on their mixed formulation,  which computes the contact forces as a Lagrange multiplier. Numerical experiments present the efficient space-time adaptive mesh refinement procedures obtained from the a posteriori error estimates. We also obtain corresponding results for elliptic problems. \\

To be specific, let  \textcolor{black}{$\Omega  \subset \mathbb{R}^n$} be a bounded, \textcolor{black}{$n$-dimensional} Lipschitz domain with boundary $\partial \Omega$. The integral fractional Laplacian $(-\Delta)^s$ of order $2s \in (0,2)$ with Dirichlet boundary conditions is defined by the bilinear form
\begin{equation*}
a(u,v) = \frac{c_{n,s}}{2} \iint_{(\Omega \times \mathbb{R}^n)\cup(\mathbb{R}^n \times \Omega)} \frac{(u(x)-u(y))(v(x)-v(y))}{\vert x - y \vert^{n+2s}} \diff y \diff x.
\end{equation*}
on \textcolor{black}{the fractional Sobolev space} $H^{s}(\Omega)$, and $c_{n,s} = \frac{2^{2s} s \Gamma\left(\frac{n+2s}{2}\right)}{\pi^{\frac{n}{2}}\Gamma\left(1-s\right)}$.
For  $f \in (H^{s}(\Omega))^*$ the associated energy is given by
\begin{equation}
E(v) = \frac{1}{2} a(v,v) - \langle f,v \rangle\ ,
\end{equation}
where $\langle \cdot, \cdot\rangle$ denotes pairing between $H^{s}(\Omega)$ and $(H^{s}(\Omega))^*$. We study the elliptic and parabolic $1$-phase  and $2$-phase problems associated to the minimization of $E$.

Given $\chi \in H^s(\Omega)$ with $\chi \leq 0$ in \textcolor{black}{$\Omega^C = \mathbb{R}^n \setminus \overline{\Omega}$}, the time-independent obstacle ($1$--phase) problem for $(-\Delta)^s$ minimizes the energy over $K_o = \lbrace v \in H^s_0(\Omega) : v \geq \chi\; \mathrm{a.e.\; in}\; \Omega \rbrace$:
\begin{pr}\label{prob:obst}
Find $u \in K_o$ such that $E(u) \leq E(v)$ for all $v \in K_o$.
\end{pr}
For $s \in (\frac{1}{2},1)$, the Signorini (thin $1$--phase) problem corresponds to an obstacle $g \in H^{s-\frac{1}{2}}(\widetilde{\Gamma})$ on a codimension $1$ subset of $\Omega$. With $K_s(\Omega) = \lbrace v \in H^s_{\widetilde{\Gamma}^C}(\Omega) : v|_{\widetilde{\Gamma}} \geq g \;\mathrm{a.e. \;on}\; \widetilde{\Gamma} \rbrace$:
\begin{pr}\label{prob:contact}
Find $u \in K_s$ such that $E(u) \leq E(v)$ for all $v \in K_s$.
\end{pr}
For friction ($2$--phase) problems, the energy contains a Lipschitz continuous functional,
\begin{align}\label{e:j's}
j_{S}(v) &= \int_{S} \mathcal{F} |v| \mathrm{d}x \ ,\ \mathcal{F} \in L^\infty(\Omega)\ , S \subset \Omega \;\text{open}\ ,\\
j_{\widetilde{\Gamma}}(v) &= \int_{\widetilde{\Gamma}} \mathcal{F} |v| \mathrm{d}s \ ,\ \mathcal{F} \in L^\infty(\widetilde{\Gamma}). \nonumber
\end{align}
\begin{pr}\label{prob:int_friction}
Find $u \in H_0^s(\Omega)$ such that
\begin{equation*}
E(u) + j_S(u) \leq E(v) + j_S(v) ,\; \forall v \in H^s_0(\Omega).
\end{equation*}
\end{pr}
The frictional contact (thin $2$--phase) problem again requires $s \in (\frac{1}{2},1)$:
\begin{pr}\label{prob:friction}
Find $u \in H^s_{\widetilde{\Gamma}^C}(\Omega)$ such that
\begin{equation*}
E(u) + j_{\widetilde{\Gamma}}(u) \leq E(v) + j_{\widetilde{\Gamma}}(v) ,\; \forall v \in H^s_{\widetilde{\Gamma}^C}(\Omega).
\end{equation*}
\end{pr}
The parabolic variants are given by the corresponding gradient flows, see Section~\ref{sec:Parabolic}.\\

In this generality, this article discusses the finite element discretizations for the associated elliptic and parabolic variational inequalities. We discuss a time-dependent discontinuous Galerkin formulation for the variational inequalities and a mixed discontinuous Galerkin formulation. In space, continuous low-order elements are used. Key results of the article present a unified error analysis for the different problems: An a priori error estimate is obtained in Theorem~\ref{l:para_apriori} (mixed), respectively Theorem~\ref{thm:Para_Falk} (variational inequality), while for $f \in L^2(\Omega)$ an a posteriori error estimate is the content of Theorem~\ref{l:para_a}, respectively Theorem~\ref{thm:para_apost_vi}. Corresponding results for the time-independent problem are presented in Section~\ref{sec:elliptic}. \\

The a posteriori error estimates lead to fast space and space-time adaptive mesh refinement procedures. Numerical experiments in Section~\ref{sec:Numerics} provide a first detailed study of these procedures for both the time-independent and time-dependent problems. For obstacle and friction problems they confirm the reliability and efficiency of the estimates and compare to discretizations on graded and uniform meshes. In the model problems, the adaptive method converges with twice the convergence rate of the uniform method.\\

For time-independent problems, adaptive methods have long been studied as fast solvers for the nonlocal boundary element formulations of contact problems \cite{gwinsteph}.  Boundary element procedures  for time-dependent problems lead to integral equations in space-time \cite{HeikoContact}, unlike the evolution equations considered in this article. An a posteriori error analysis for nonlocal obstacle problems has been studied in \cite{Nochetto}, based on the associated variational inequality. The implementation and adaptive methods for fractional Laplace equations are considered in \cite{Acosta1,Glusa,d2013fractional}, without contact. Furthermore, spectral nonlocal operators have been of interest \cite{SalgadoOtarola2016}. For motivations from continuum mechanics see \cite{DuICM}.\\

\noindent \textit{The article is organized in the following way:} Section~\ref{sec:Definitions} recalls the definitions and notation related to the fractional Laplacian as well as the suitable Sobolev spaces. Section~\ref{s:VI} discusses the nonlocal variational inequalities and establishes the equivalence of the weak and strong formulations. Section~\ref{sec:discretization} describes the discretization. The a priori and a posteriori error analysis for the time-independent problems is presented in Section~\ref{sec:elliptic}. The error analysis for the time-dependent problems is the content of Section~\ref{sec:Parabolic}. After discussing implementational challenges in Section~\ref{sec:implementation},  in Section~\ref{sec:Numerics} we present numerical experiments on uniform, graded and space-time adaptive meshes based on the mixed formulation.

\section{Preliminaries}\label{sec:Definitions}

This section recalls some notation and some basic properties related to the fractional Laplacian $(-\Delta)^s$, $0<s<1$, in a bounded Lipschitz domain $\Omega \subset \mathbb{R}^n$.

The fractional Sobolev space $H^s(\Omega)$ is defined by \textcolor{black}{\cite{gwinsteph}}
\begin{equation*}
H^s(\Omega) = \lbrace v \in L^2(\Omega): |v|_{H^s(\Omega)} < \infty \rbrace,
\end{equation*}
where $|\cdot|_{H^s(\Omega)}$ is the Aronszajn-Slobodeckij seminorm
\begin{equation*}
|v|_{H^s(\Omega)}^2 = \iint_{\Omega \times \Omega} \frac{(v(x)-v(y))^2}{\vert x - y \vert^{n+2s}} \diff y \diff x.
\end{equation*}
$H^s(\Omega)$ is a Hilbert space endowed with the norm
\begin{equation*}
\|v\|_{H^s(\Omega)} = \|v\|_{L^2(\Omega)} + |v|_{H^s(\Omega)}.
\end{equation*}
The closure in $H^s(\Omega)$ of the subspace of functions $v \in C^\infty(\overline{\Omega})$ with $v = 0$ on $\partial \Omega \setminus \widetilde{\Gamma}$ is denoted by $H^s_{\widetilde{\Gamma}^C}(\Omega)$. As common, we write $H^s_0(\Omega) = H^s_{\partial \Omega}(\Omega)$.
The dual space of \textcolor{black}{$H^s_{\widetilde{\Gamma}}(\Omega)$} is denoted by $H^{-s}_{\widetilde{\Gamma}^C}(\Omega)$, where $\widetilde{\Gamma}^C = \Gamma \setminus \overline{\widetilde{\Gamma}}$, and we denote the duality pairing by $\langle \cdot,\cdot\rangle$.\\

On $H^s(\Omega)$ we define a bilinear form as
\begin{equation*}
a(u,v) = \frac{c_{n,s}}{2} \iint_{(\Omega \times\mathbb{R}^n)\cup(\mathbb{R}^n \times \Omega)} \frac{(u(x)-u(y))(v(x)-v(y))}{\vert x - y \vert^{n+2s}} \diff y \diff x\ ,
\end{equation*}
where $c_{n,s} = \frac{2^{2s} s \Gamma\left(\frac{n+2s}{2}\right)}{\pi^{\frac{n}{2}}\Gamma\left(1-s\right)}$. It is continuous and coercive: There exist $C,\alpha > 0$ such that
\begin{align}
a(u,v) \leq C \|u\|_{H^s}\|v\|_{H^{s}}\ ,\qquad a(u,u) \geq \alpha \|u\|_{H^s}^2\ ,
\end{align}
for all $u, v \in H^s(\Omega)$. Note that the fractional Laplacian $(-\Delta)^s$ with homogeneous Dirichlet boundary conditions is the operator associated to the bilinear form $a(u,v)$. For more general boundary conditions see \cite{dipierro2017nonlocal}. By the Lax-Milgram lemma, $(-\Delta)^s: H^s_0(\Omega) \to H^{-s}(\Omega)$ is an isomorphism.

For sufficiently smooth functions $u: \mathbb{R}^n \to \mathbb{R}$, $(-\Delta)^s u(x)$ is given by
\begin{equation}\label{e:fl}
(-\Delta)^s u (x) = c_{n,s}\, P.V. \int_{\mathbb{R}^n} \frac{u(x)-u(y)}{\vert x - y \vert^{n+2s}} \diff y,
\end{equation}
where $P.V.$ denotes the Cauchy principle value. $(-\Delta)^s$ may also be defined in terms of the Fourier transform $\mathcal{F}$: $\mathcal{F}((-\Delta)^s u)\textcolor{black}{(\xi)} = \vert\xi \vert^{2s} \mathcal{F}u(\xi)$. This formulation shows that the ordinary Laplacian is recovered for $s = 1$, \textcolor{black}{which may be less obvious from the bilinear form $a$ or the integral expression \eqref{e:fl}.} \\

Let $H$ be a Hilbert space corresponding to $H^s_0(\Omega)$ or $H^s_{\widetilde{\Gamma}}(\Omega)$ for respective problems and let $H^{*}$ be the dual space of $H$.\\

For the analysis of time-dependent problems, the Bochner spaces $\mathbb{W}(0,T)$ prove useful:
\begin{equation*}
\mathbb{W}(0,T) = \lbrace v \in L^2(0,T;H) : \partial_t v \in L^2(0,T;H^{*})\rbrace\ .
\end{equation*}
It is a Hilbert space with norm
\begin{equation*}
\|v\|^2_{\mathbb{W}(0,T)} = \int_0^T \left\{ \|\partial_t v\|^2_{H^{*}} + \|v\|^2_{H}\right\} \dt + \|v(T)\|^2_{L^2(\Omega)}
\end{equation*}
and continuously embeds into $C^0(0,T;L^2(\Omega))$.\\

Furthermore, let $\langle\cdot,\cdot\rangle$ denote the pairing between $H$ and $H^*$, or $H^{\frac{1}{2}-s}$ and $H^{s-\frac{1}{2}}$ where appropriate.
\section{Elliptic and parabolic variational inequalities}\label{s:VI}
In this section we introduce a large class of elliptic and parabolic variational inequalities associated with the fractional Laplacian.
\subsection{Variational inequality formulation}
Let $a(\cdot,\cdot)$ be the bilinear form associated with the fractional Laplacian, and $K \subseteq H$, be a closed convex subset. We then consider the following problems:
\begin{problem}\label{p:ell1}
Find $u \in K$ with $f \in H^{*}$ such that
\begin{equation*}
a(u,v-u) \geq \langle f,v-u\rangle, \mathrm{for\;all}\; v \in K.
\end{equation*}
\end{problem} 
Given $j: H \to \mathbb{R}$ convex and lower semi-continuous, we further consider:
\begin{problem}\label{p:ell2}
Find $u \in H$, with $f \in H^{*}$ such that
\begin{equation*}
a(u,v-u) +j(v) - j(u) \geq \langle f,v-u\rangle, \mathrm{for\;all}\; v \in H.
\end{equation*}
\end{problem}
The friction functionals $j_S$ and $j_{\widetilde{\Gamma}}$ from \eqref{e:j's} are of particular interest.
Existence and uniqueness to both Problems \ref{p:ell1} and \ref{p:ell2} can be found in \cite[Theorem~2.1]{KinderlehrerStampacchia}, \cite[Theorems~2.1, 2.2]{glowinski}.\\
\\
Also for time-dependent variational inequalities, $K$ denotes a nonempty closed convex subset of $H$. We define $\mathbb{W}_K(0,T) = \lbrace v \in \mathbb{W}(0,T): v(t) \in K, \mathrm{a.e.\; in}\; t \in (0,T)\rbrace$. For a given initial condition $u_0 \in K$ we obtain the  problem:
\begin{problem}\label{p:para1}
Find $u \in \mathbb{W}_K(0,T)$ with $f \in L^2(0,T;H^{*})$ and $u(0) = u_0$ such that 
\begin{equation*}
\int_0^T \langle \partial_t u, v-u\rangle +a(u,v-u) \dt \geq \int_0^T \langle f,v-u\rangle \dt \;\mathrm{for\;all\;} v \in W_K(0,T).
\end{equation*}
\end{problem}
Furthermore, let $j(\cdot)$ be convex, lower semi-continuous and integrable for all $v \in \mathbb{W}_K(0,T)$. 
\begin{problem}\label{p:para2}
Find $u \in \mathbb{W}(0,T)$ with $f \in L^2(0,T;H^{*})$, $u(0) = u_0$, and $j(u_0) < \infty$ such that 
\begin{equation*}
\int_0^T \langle \partial_t u, v-u\rangle +a(u,v-u) \dt +\int_0^T j(v) - j(u) \dt \geq \int_0^T \langle f,v-u\rangle \dt \;\mathrm{for\;all\;} v \in \mathbb{W}(0,T).
\end{equation*}
\end{problem}
Existence and uniqueness of solutions to Problems \ref{p:para1} and \ref{p:para2} follows from \cite[Ch.1~Section~(5.2)]{DuvautLions} and \cite[Ch.2~Section~(2.1)]{Brezis1}. In later sections we provide a unified treatment of the above problems, both in the time-independent and time-dependent case.\\

\subsection{Strong formulations}

For the a posteriori error estimates derived later in this work, the strong formulation of Problems~\ref{prob:obst}--\ref{prob:friction} proves relevant. Problems~\ref{prob:obst} and \ref{prob:contact} correspond to the weak formulation in Problem~\ref{p:ell1}, and we refer to  \cite{Silvestre} for a discussion of their strong formulations:

\begin{pblm}[strong form of Problem \ref{prob:obst}]\label{pr:obst}
Find $u$ such that
\begin{equation}
\begin{aligned}
(-\Delta)^s u &\geq f  &\mathrm{in} \; \Omega\\
u &= 0 &\mathrm{in}\; \Omega^C \\
u \geq \chi ,\;
(u-\chi)((-\Delta)^s u - f) &= 0 &\mathrm{in} \; \Omega.
\end{aligned}
\end{equation}
\end{pblm}

\begin{pblm}[strong form of Problem \ref{prob:contact}]\label{pr:contact}
Find $u$ such that
\begin{equation}
\begin{aligned}
(-\Delta)^s u &\geq f  &\mathrm{in} \; \Omega\\
u &= 0   &\mathrm{in}\; \Omega^C \\
u \geq g,\; (u - g)\sigma(u) &= 0 &\mathrm{on}\; \widetilde{\Gamma}.
\end{aligned}
\end{equation}
Here $(-\Delta)^s u -f = \sigma(u) \delta_{\widetilde{\Gamma}}$ defines a unique function $\sigma(u)$ on $\widetilde{\Gamma}$.
\end{pblm}

We now consider the strong formulation of the friction problems~\ref{prob:int_friction} and \ref{prob:friction}, corresponding to the weak formulation in Problem~\ref{p:ell2}.  They read:
\begin{pblm}[strong form of Problem \ref{prob:int_friction}]\label{pr:int_friction}
Find $u$ such that
\begin{equation}
\begin{aligned}
(-\Delta)^s u &\leq f &\mathrm{in} \; \Omega\\
u &= 0 &\mathrm{in}\; \Omega^C\\
|(-\Delta)^s u - f| \leq \mathcal{F}, \;  u ((-\Delta)^s u-f) + \mathcal{F}|u| &=0 &\textcolor{black}{\mathrm{in}\; S \subseteq \Omega.}
\end{aligned}
\end{equation}
\end{pblm}

\begin{pblm}[strong form of Problem \ref{prob:friction}]\label{pr:friction} 
Find $u$ such that
\begin{equation}
\begin{aligned}
(-\Delta)^s u &\leq f &\mathrm{in} \; \Omega\\
u &= 0 &\mathrm{in}\; \Omega^C \\
|\sigma(u)| \leq \mathcal{F}, \;  u \sigma(u) + \mathcal{F}|u| &=0 &\mathrm{on}\; \widetilde{\Gamma}.
\end{aligned}
\end{equation}
Here $(-\Delta)^s u -f = \sigma(u) \delta_{\widetilde{\Gamma}}$ defines a unique function $\sigma(u)$ on $\widetilde{\Gamma}$.
\end{pblm}

\begin{lem}
a) If $u$ is a sufficiently smooth solution to Problem \ref{prob:int_friction}, then $u$ satisfies Problem \ref{pr:int_friction}. Conversely, a solution to Problem \ref{pr:int_friction} is also a solution to Problem \ref{prob:int_friction}. \\
b) If $u$ is a sufficiently smooth solution to Problem \ref{prob:friction}, then $u$ satisfies Problem \ref{pr:friction}. Conversely, a solution to Problem \ref{pr:friction} is also a solution to Problem \ref{prob:friction}. 
\end{lem}
\begin{proof}
a) We use the formulation of Problem \ref{prob:int_friction} as a variational inequality, i.e.~Problem~\ref{p:ell2}.
Substituting $\xi = v-u$ into Problem~\ref{p:ell2} for any $\xi \in H^s_0(\Omega)$, we find
\begin{equation*}
a(u,\xi) +j(u+\xi) - j(u) \geq \langle f,\xi\rangle\ .
\end{equation*}
Let $\xi \in C_c^{\infty}(\Omega)$ such that $\xi$ vanishes on the set $C=\{x \in S : u(x) = 0\}$. Replacing $\xi$ with $-\xi$ we obtain
\begin{equation*}
a(u,\xi) -j(u-\xi) + j(u) \leq \langle f, \xi \rangle.
\end{equation*}
For $|\xi| < |u|$ from the definition of $j$ we see that $$j(u) = \frac{1}{2}\left(j(u+\xi)+j(u-\xi)\right) \ .$$  
It follows that
\begin{equation*}
a(u,\xi) = \langle f, \xi \rangle ,
\end{equation*}
or $(-\Delta)^s u = f$ in $\Omega \setminus \overline{C}$. Furthermore, for arbitrary $\xi$ we have from the definition of $j$  and the triangle inequality:
\begin{equation*}
a(u,\xi) +j(\xi) - \langle f, \xi \rangle \geq 0.
\end{equation*}
Since $\mathcal{F}\geq 0$, using this inequality with $\xi$ and $-\xi$ we conclude
\begin{equation*}
|a(u,\xi) - \langle f, \xi\rangle| \leq j(\xi) = \int_S\mathcal{F} |\xi|.
\end{equation*}
As this holds for all $\xi$, the first asserted inequality in $S$ follows. Finally, to deduce the complementarity condition in $S$, we use $\xi = \pm u$ to obtain
\begin{equation*}
a(u,u) + j(u) -\langle f, u \rangle = 0.
\end{equation*}
Therefore
\begin{equation*}
\int_{C} u ((-\Delta)^s u -f) + \mathcal{F} |u|\; \mathrm{d}x = 0.
\end{equation*}
As the integrand is non-negative, it vanishes almost everywhere. This concludes the first part fo the assertion.\\

We now show that a solution of the strong Problem~\ref{pr:int_friction} satisfies the weak formulation. To do so, multiply the first equation of Problem~\ref{pr:int_friction} with a test function $v \in H^s_0(\Omega)$ and integrate over $\Omega$:
\begin{equation}
\int_{\Omega} v (-\Delta)^s u\; \mathrm{d} x \geq \int_{\Omega} f v \;\mathrm{d}x.
\end{equation}
Using
\begin{equation}
a(u,v) - \langle f,v \rangle  = \int_{C} ((-\Delta)^s u - f) v\; \mathrm{d} x,
\end{equation}
we see that
\begin{equation}
a(u,v-u) - \langle f, v-u \rangle +j(v) - j(u) = \int_C ((-\Delta)^s u - f) (v-u) + \mathcal{F}(|v|-|u|)\; \mathrm{d}x = I.
\end{equation}
It remains to show that the right hand side is nonnegative. If $|(-\Delta)^s u -f| < \mathcal{F}$ in a point $x$, then by the contact condition $u(x) = 0$. Thus,
\begin{equation*}
\begin{split}
I = \int_C ((-\Delta)^s u - f) (v-u) + \mathcal{F}(|v|-|u|)\; \mathrm{d}x &= \int_C ((-\Delta)^s u - f) (v) + \mathcal{F}|v|\; \mathrm{d}x \\ &\geq \int_C -|(-\Delta)^s u - f||v| + \mathcal{F}|v| \;\mathrm{d}x \\ &\geq 0.
\end{split}
\end{equation*}
If $|(-\Delta)^s u -f| = \mathcal{F}$ in $x$, then there exists $\mu \geq 0$ such that $u = - \mu ((-\Delta)^s u - f)$. Therefore,
\begin{equation*}
\begin{split}
I &= \int_C ((-\Delta)^s u - f)v + \mathcal{F}|v| + \mu |(-\Delta)^s u - f|^2  - \mu \mathcal{F} |(-\Delta)^s u - f|\mathrm{d}x \\ &\geq \int_C -|(-\Delta)^s u - f||v| + \mathcal{F}|v| \mathrm{d}x \\ &\geq 0.
\end{split}
\end{equation*}
We conclude that $u$ satisfies Problem \ref{p:ell2}. \\
\noindent b) The corresponding proof for the frictional contact Problem \ref{pr:friction} is analogous to part a), with the additional key observation:
\begin{lem}
Assume $\Lambda : H^s(\Omega) \to \mathbb{C}$ is continuous and $\Lambda(\varphi) = \Lambda(\psi)$ whenever $\varphi|_{\widetilde{\Gamma}} = \psi|_{\widetilde{\Gamma}}$. Then there exists a unique, continuous $\lambda : H^{\frac{1}{2}-s}_0(\widetilde{\Gamma}) \to \mathbb{C}$ such that 
\begin{equation*}
\lambda(\varphi) = \Lambda(E \varphi),
\end{equation*}
where $E: H^{s-\frac{1}{2}}_0(\widetilde{\Gamma}) \to H^s(\Omega)$ is any extension operator.
\end{lem}
\textcolor{black}{\begin{proof}
As a composition of continuous maps, $\lambda(\varphi) = \Lambda(E \varphi)$ defines a continuous functional $\lambda$ on $H^{\frac{1}{2}-s}_0(\widetilde{\Gamma})$.\\
Concerning uniqueness, if $E_1, E_2: H^{s-\frac{1}{2}}_0(\widetilde{\Gamma}) \to H^s(\Omega)$ are two extension operators, $(E_1 \varphi)|_{\widetilde{\Gamma}} = (E_2 \varphi)|_{\widetilde{\Gamma}} = \varphi$. By assumption on $\Lambda$, we conclude $\Lambda(E_1 \varphi) = \Lambda(E_2 \varphi)$, so that $\Lambda \circ E_1 = \Lambda \circ E_2$ defines a unique $\lambda : H^{\frac{1}{2}-s}_0(\widetilde{\Gamma}) \to \mathbb{C}$. 
\end{proof}}
From the Lemma one obtains that the distribution $(-\Delta)^s u -f$ is supported on $\widetilde{\Gamma}$. In order to belong to $H^{\frac{1}{2}-s}_0(\widetilde{\Gamma})$, it must be proportional to $\delta_{\widetilde{\Gamma}}$.
\end{proof}

\section{Discretization}\label{sec:discretization}
For simplicity of notation, we assume that $\Omega$ has a polygonal boundary. Let $\mathcal{T}_h$ be a family of shape-regular triangulations of $\Omega$ and $\mathbb{V}_h\subset H^s_0(\Omega)$, $s \in (0,1)$, the associated space of continuous piecewise linear functions on $\mathcal{T}_h$, vanishing at the boundary. Furthermore, let $\mathbb{V}_h^{\widetilde{\Gamma}^C} \subset H^s(\Omega)$ be the space of continuous piecewise linear functions on $\mathcal{T}_h$, vanishing on $\widetilde{\Gamma}^C$. Let $\mathbb{M}_H$ be the space of piecewise constant functions on $\mathcal{T}_H$. We denote the set of nodes of $\mathcal{T}_h$ by $\mathcal{P}_h$ (including the boundary nodes) and the nodal basis of $\mathbb{V}_h$ (resp. $\mathbb{V}_h^{\widetilde{\Gamma}^C}$) by $\lbrace \phi_i\rbrace$. Let $S_i$ be the support of the piecewise linear hat function associated to node $i$.\\

Let $K_h$ be the discrete counterpart of $K$. That is, for Problems~\ref{prob:obst} and \ref{prob:contact},
\begin{equation}
K_{oh} := (K_{o}(\Omega))_h = \lbrace v_h \in \mathbb{V}_h : v_h \geq \chi_h\; a.e.\; in\; \Omega \rbrace,
\end{equation}
and
\begin{equation}
K_{sh} :=(K_{s}(\Omega))_h = \lbrace v_h \in \mathbb{V}_h^{\widetilde{\Gamma}^C} : v_h|_{\widetilde{\Gamma}} \geq g_h \;a.e. \; on\; \widetilde{\Gamma} \rbrace.
\end{equation}

The set $K_h$ is nonempty, closed and convex.  To simplify the presentation, we assume a conforming discretization $K_h \subset K$. In the case of the obstacle problem this holds if $\chi \in \mathbb{V}_h$, while for the thin obstacle problem this holds if $g_h$ is the restriction to $\widetilde{\Gamma}$ of a function in $\mathbb{V}_h$. The appropriate spaces of restricted function on $\widetilde{\Gamma}$ are denoted by ${\widetilde{\mathbb{V}}_h}$ or ${\widetilde{\mathbb{M}}_H}$. See \cite{Veeser} for the adaptations necessary for nonconforming discretizations.\\
For the time discretization we consider a decomposition of the time interval $I = [0,T]$ into subintervals $I_k = [t_{k-1}, t_{k})$ with time step $\tau_k$. The associated space of piecewise polynomial functions  of degree $q = 0,1$  is denoted by $\mathbb{T}_{\tau}$. We  define $\mathbb{W}_{h\tau}(0,T) = \mathbb{V}_h \otimes \mathbb{T}_{\tau}$ and $\mathbb{M}_{H\tau}(0,T) = \mathbb{M}_H \otimes \mathbb{T}_{\tau}$. For the adaptive computations, also local time steps are considered. We denote these discrete local-in-time spaces by $\widetilde{\mathbb{W}}_{h\tau}(0,T)$, respectively $\widetilde{\mathbb{M}}_{H\tau}(0,T)$. Similarly to the time-independent case, let $\widetilde{\mathbb{W}}_{h\tau}^{\widetilde{\Gamma}}(0,T)$ and $\widetilde{\mathbb{M}}_{H\tau}^{\widetilde{\Gamma}}(0,T)$ be the  spaces of discrete functions vanishing on $\widetilde{\Gamma}^C$.\\

The discrete elliptic problems associated to the two classes of variational inequalities~\ref{p:ell1} and \ref{p:ell2} are:
\begin{problem}\label{p:disc_ell1}
Find $u_h \in K_h$ such that for all $v_h \in K_h$,
\begin{equation*}\label{p:disc_ell}
a(u_h, v_h - u_h) \geq \langle f, v_h - u_h\rangle.
\end{equation*}
\end{problem}
\begin{problem}\label{p:disc_ell2}
Find $u_h \in \mathbb{V}_h$ such that for all $v_h \in \mathbb{V}_h$,
\begin{equation*}\label{p:disc_ell_friction}
a(u_h, v_h - u_h) +j(v_h) - j(u_h) \geq \langle f, v_h - u_h\rangle.
\end{equation*}
\end{problem}

For the discrete parabolic problem we introduce the space-time bilinear form $B_{DG}(\cdot,\cdot)$ given by
\begin{equation}\label{e:DGbilinear}
B_{DG}(u,v) = \sum_{k=1}^M \int_{I_k} \langle \partial_t u,v \rangle + a(u,v) \dt.
\end{equation}
Similarly as in the elliptic case, the discrete parabolic obstacle problem associated with Problem~\ref{p:para1} is given by:
\begin{problem}\label{p:disc_para1}
Find $u_{h\tau} \in K_{h\tau}$ such that for all $v_{h\tau} \in K_{h\tau}$,
\begin{equation*} \label{p:disc_para}
\begin{split}
B_{DG}(u_{h\tau},v_{h\tau}-u_{h\tau}) &+ \sum_{k=1}^M \langle \left[ u_{h\tau}\right]^{k-1},(v_{h\tau}-u_{h\tau})_+\rangle \\ &  \geq \sum_{k=1}^M \int_{I_k} \langle f, v_{h\tau} - u_{h\tau} \rangle \dt.
\end{split}
\end{equation*}
\end{problem}
Here $v^n_+ = \lim_{s \to 0^+} u(t^n + s)$ and $[v]^n = v^n_+-v^n$. As the obstacle is assumed to be independent of time, the convex subset $K_{h\tau}$ is defined in a similar manner as $K_h$.
The discretization of parabolic friction problems associated with Problem~\ref{p:disc_para2} reads:
\begin{problem}\label{p:disc_para2}
Find $u_{h\tau} \in \mathbb{W}_{h\tau}(0,T)$ such that for all $v_{h\tau} \in \mathbb{W}_{h\tau}(0,T)$,
\begin{equation*} \label{p:disc_para_fric}
\begin{split}
B_{DG}(u_{h\tau},v_{h\tau}-u_{h\tau}) + \sum_{k=1}^M \int_{I_k} j(v_{h\tau}) -j(u_{h\tau}) \dt &+ \sum_{k=1}^M \langle \left[ u_{h\tau}\right]^{k-1},(v_{h\tau}-u_{h\tau})_+\rangle \\ &  \geq \sum_{k=1}^M \int_{I_k} \langle f, v_{h\tau} - u_{h\tau} \rangle \dt.
\end{split}
\end{equation*}
\end{problem}
We conclude the section with a variant of \cite[Lemma~2.7]{SchotzauSchwab} adapted to fractional operators. It establishes the coercivity of the bilinear form $B_{DG}(\cdot,\cdot)$ in combination with the jump terms. Note that the proof in \cite{SchotzauSchwab} uses only the coercivity of the bilinear form $a(\cdot,\cdot)$ and therefore applies to both local and nonlocal problems, in the appropriate function spaces.
\begin{lem}\label{l:coercive}
Let $B_{DG}(\cdot,\cdot)$ be defined as in~\eqref{e:DGbilinear}. Let $v_{h\tau} \in \mathbb{W}_{h\tau}(0,T)$. Then,
\begin{equation*}\label{e:coercive}
\begin{split}
B_{DG}(&v_{h\tau},v_{h\tau}) + \sum_{k=1}^M \langle [v_{h\tau}]^{k-1}, v_{h\tau +}^{k-1}\rangle  \geq \\ &\alpha \|v_{h\tau}\|^2_{L^2(0,T;H)} + \frac{1}{2}\|v_{h\tau +}^0\|_{L^2(\Omega)}^2 + \frac{1}{2}\sum_{k=1}^{M-1} \|[v_{h\tau}]^k\|_{L^2(\Omega)}^2 +\frac{1}{2}\|v_{h\tau -}^M\|_{L^2(\Omega)}^2.
\end{split}
\end{equation*}
\end{lem}

\section{Elliptic problems}\label{sec:elliptic}
In this section we discuss the error analysis of elliptic variational inequalities introduced in Section~\ref{s:VI}. We address a priori and a posteriori error estimates for such problems both for the variational inequality and a mixed formulation. Combined with known regularity results the a priori estimates allow us to deduce convergence rates for the specific problems introduced in Section~\ref{intro}.
\subsection{A priori error estimates for variational inequalities}
We first discuss a priori error estimates for fractional elliptic variational inequalities  corresponding to contact problems in the domain. Corresponding results for the thin problems can be derived analogously. Observe an analogue of Falk's lemma for elliptic variational inequalities \cite{falk}, as adapted to Problem~\ref{p:ell1}:

\begin{lem}\label{l:falk_ell1}
Let $u \in K$ and $u_h \in K_h$ be solutions of Problem~\ref{p:ell1} and \ref{p:disc_ell1}, respectively. Then,
\begin{equation*} \label{e:ell_falk}
\begin{split}
\|u-u_h\|_{H}^2 \lesssim& \inf_{v \in K} \lbrace \|f-(-\Delta)^s u\|_{H^{*}} \|u_h-v\|_{H} \rbrace \\&+ \inf_{v_h \in K_h} \lbrace \|f-(-\Delta)^s u\|_{H^*} \|u-v_h\|_{H} + \|u-v_h\|_{H}^2   \rbrace.
\end{split}
\end{equation*}
\end{lem}
Accounting for $j(\cdot)$, a similar result holds for Problem~\ref{p:ell2}.

\begin{lem}\label{l:falk_ell2}
Let $u \in H$ and $u_h \in \mathbb{V}_h$ be solutions of Problem~\ref{p:ell2} and \ref{p:disc_ell2}, respectively. Let $j(\cdot)$ be a proper, convex, l.s.c. functional on $H$. Then,
\begin{equation*}
\begin{split}
\|u-u_h\|_{H}^2 \lesssim& \inf_{v \in H} \lbrace \|f-(-\Delta)^s u\|_{H^*} \|u_h-v\|_{H} +  j(u_h) - j(v) \rbrace \\&+ \inf_{v_h \in \mathbb{V}_h} \lbrace \|f-(-\Delta)^s u\|_{H^*} \|u-v_h\|_{H} +  j(u) - j(v_h)  + \|u-v_h\|_{H}^2   \rbrace.
\end{split}
\end{equation*}
\end{lem}

\begin{rem}
If equality holds in the variational inequality, the residual $f - (-\Delta)^s u$ vanishes and we recover Cea's lemma. In the general case, $f - (-\Delta)^s u$ does not vanish and the convergence rate reduces by a factor $2$.
Since we assume that $K_h \subset K$, $\mathbb{V}_h \subset H$, we have the internal approximation of the variational inequalities, thus we can choose $v = u_h$ and so the first infimum in Lemma~\ref{l:falk_ell1} and Lemma~\ref{l:falk_ell2} vanishes.
\end{rem}

We briefly discuss explicit convergence rates for the elliptic problems. Under the assumption that $u \in H_0^{\ell}(\Omega)$ for some $\ell >s$, we can use standard interpolation argument to establish a convergence rate of discrete solution. Note that for the obstacle problem as defined in Problem~\ref{prob:obst}, $K_o = \lbrace v \in H_0^s(\Omega): v \geq \chi\rbrace$ and $K_{oh} = \lbrace v_h \in \mathbb{V}_h: v_h \geq \chi\rbrace$.
\begin{lem}
Let $f \in L^2(\Omega) \cap \mathbb{V}_h$ and $\chi \in \mathbb{V}_h$. Let $u \in K_o$ and $u_h \in K_{oh}$ be solutions of Problem~\ref{p:ell1}, \ref{p:disc_ell1}. Then,
\begin{equation}\label{e:obst_conv_rate}
\|u-u_h \|_{H_0^s(\Omega)} \leq C h^{\frac{\ell-s}{2}} \|u\|_{H^\ell(\Omega)},
\end{equation}
for $0<s<1 \leq \ell$.
\end{lem}

\begin{rem} \label{regremark}
Provided that $f \in H^r_0(\Omega)$ for some $r\geq -s$ and $\partial\Omega \in C^{\infty}$, the solution $u \in H_0^s(\Omega)$ of the unconstrained problem
\begin{equation*}
a(u,v) = \langle f,v \rangle,\;\; v\in H_0^s(\Omega),
\end{equation*}
belongs to
\begin{equation*}
u \in \begin{cases}
H^{2s+r}, & s+r \leq \frac{1}{2}\\
H^{s+\frac{1}{2}-\varepsilon}, & s+r \geq \frac{1}{2}
\end{cases}
\end{equation*}
This implies that for $f \in L^2(\Omega)$, we may expect the solution $u$ to have up to $s+\frac{1}{2}$ derivatives in $L^2(\Omega)$. We conclude from the estimate in \eqref{e:obst_conv_rate} that $\|u-u_h \|_{H_0^s(\Omega)} \leq \mathcal{O}(h^{1/4-\varepsilon})$. The smoothness of the solution is limited by the behavior near the Dirichlet boundary $\partial \Omega$, where $u(x) \sim d(x,\partial \Omega)$. This behaviour has been exploited in \cite{Acosta1} who showed that the solution  admits $1+s-\varepsilon$ derivatives in an appropriate  weighted Sobolev space. For further discussion of the expected regularity of solutions of variational inequalities, see \cite{Brezis1}, as well as \cite{bls, Salgado} for refined estimates in the case of the nonlocal obstacle problem. 
\end{rem}

\subsection{A posteriori error estimate for variational inequalities}
In this section we discuss a posteriori error estimates of elliptic variational inequalities in Problems~\ref{p:ell1} and \ref{p:ell2}. We provide a careful analysis of Problems~\ref{prob:obst} and \ref{prob:int_friction} with contact in the interior of the domain $\Omega$, so that corresponding bounds for the thin contact Problems~\ref{prob:contact} and \ref{prob:friction} readily follow. For simplicity, we consider data $f,\;\chi \in \mathbb{V}_h$ in the finite element space; for the modifications to general $f, \; \chi$ see \cite{Nochetto, Veeser}.\\
Consider Problem \ref{prob:obst}. We define the Lagrange multiplier $\sigma \in H^*$ as 
\begin{equation}\label{e:sigma1}
\langle \sigma, v  \rangle = \langle f, v \rangle - a(u, v),\; \forall v \in H.
\end{equation}
Let $\sigma_h \in \mathbb{V}_h$ be the discrete Lagrange multiplier defined by $\langle \sigma_h, v_h\rangle = \langle f, v_h \rangle - a(u_h,v_h)$ for $v_h \in \mathbb{V}_h$. Also let $r_h = f - (-\Delta)^s u_h$. Then the following result holds:
\begin{lem}[Obstacle problem]\label{l:Apost_obst}
Let $u,\; u_h$ be solutions of Problem~\ref{p:ell1} and \ref{p:disc_ell1}, respectively, associated with Problem~\ref{prob:obst}. Assume that $f \in \mathbb{V}_h$ and $\chi \in \mathbb{V}_h$. Then,
\begin{equation*}
\|u-u_h\|^2_{H_0^s(\Omega)} + \|\sigma-\sigma_h\|^2_{H^{-s}(\Omega)} \lesssim \|r_h-\sigma_h\|^2_{H^{-s}(\Omega)} - \langle \sigma_h,u_h-\chi \rangle.
\end{equation*}
\end{lem}
\begin{proof}
By definition of $\sigma$ and $\sigma_h$ the following equality holds
\begin{equation}\label{e:residual}
a(u-u_h,v) + \langle \sigma - \sigma_h,v \rangle  = \langle f, v \rangle -a(u_h,v) - \langle \sigma_h,v \rangle = \langle r_h, v \rangle - \langle \sigma_h,v \rangle.
\end{equation}
Choosing $v = u-u_h$ in \eqref{e:residual}, we obtain
\begin{equation*}
\| u-u_h\|_{H_0^s(\Omega)}^2 \leq \frac{1}{2}\|r_h-\sigma_h\|_{H^{-s}(\Omega)}^2 + \frac{1}{2}\| u-u_h\|_{H_0^s(\Omega)}^2 - \langle \sigma - \sigma_h,u-u_h \rangle.
\end{equation*}
We note that for any $w\in H_0^s(\Omega)$
\begin{equation*}
\langle \sigma - \sigma_h,w \rangle = a(u_h-u,w) +\langle r_h-\sigma_h,w \rangle,
\end{equation*}
which leads to estimate
\begin{equation}\label{e:ellapost_midway}
\|u-u_h\|^2_{H_0^s(\Omega)} + \|\sigma-\sigma_h\|^2_{H^{-s}(\Omega)} \lesssim \|r_h-\sigma_h\|^2_{H^{-s}(\Omega)} - \langle \sigma - \sigma_h,u-u_h \rangle.
\end{equation}
To determine a computable bound for the second term, for the obstacle problem here, we note that $\langle \sigma, u-u_h\rangle \geq 0$. In addition,
\begin{equation*}
\begin{split}
\langle \sigma_h,u-u_h \rangle  &= \langle \sigma_h,u-\chi\rangle - \langle \sigma_h,u_h-\chi \rangle  \\ &\leq  - \langle \sigma_h,u_h-\chi \rangle.
\end{split}
\end{equation*}
The result follows by combining the above estimates.
\end{proof}

\begin{rem}
The thin obstacle problem falls into the same framework. Here, the convex set $K \subset {H}^s_{\widetilde{\Gamma}^C}(\Omega)$ is replaced by
\begin{equation}
K_s(\Omega) = \lbrace v \in H^s_{\widetilde{\Gamma}^C}(\Omega) : v|_{\widetilde{\Gamma}} \geq g \;\mathrm{a.e.\;on}\; \widetilde{\Gamma} \rbrace.
\end{equation}
Estimate \eqref{e:ellapost_midway} then holds verbatim, if the second term on the right hand side is taken on $\widetilde{\Gamma}$. Note that $\langle \sigma, u|_{\widetilde{\Gamma}} - g \rangle = 0$ and $u|_{\widetilde{\Gamma}} -g \geq 0$ almost everywhere on $\widetilde{\Gamma}$, so that
\begin{equation}
\langle \sigma_h - \sigma,u|_{\widetilde{\Gamma}}-u_h|_{\widetilde{\Gamma}} \rangle \leq \langle \sigma_h, g- u_h|_{\widetilde{\Gamma}} \rangle.
\end{equation}
\end{rem}
This implies:
\begin{lem}[Signorini problem]
Let $u,\;u_h$ be solutions of Problems~\ref{p:ell1} and \ref{p:disc_ell1}, respectively, associated with Problem~\ref{prob:contact}. Assume that $f \in \mathbb{V}_h$ and $g \in \mathbb{V}_h$. Then,
\begin{equation*}
\|u-u_h\|^2_{{H}^s_{\widetilde{\Gamma}^C}(\Omega)} + \|\sigma-\sigma_h\|^2_{H^{-s}(\Omega)} \lesssim \|r_h-\sigma_h\|^2_{H^{-s+\frac{1}{2}}(\widetilde{\Gamma})} - \langle \sigma_h,g - u_h|_{\widetilde{\Gamma}} \rangle.
\end{equation*}
\end{lem}
\medskip

For the interior friction problem, let $\sigma$ be defined as in Equation~\eqref{e:sigma1} and let $\sigma_h$ to be the discrete counterpart of $\sigma$.
\begin{lem}[Interior friction problem]
Let $u,u_h$ be solutions of Problems~\ref{p:ell2} and \ref{p:disc_ell2}, respectively, associated with Problem~\ref{prob:int_friction}. Assume that $f \in \mathbb{V}_h$. Then,
\begin{equation*}
\begin{split}
\|u-u_h\|_{H^s_0(\Omega)}^2 + \|\sigma-\sigma_h\|_{H^{-s}(\Omega)}^2 \lesssim &\|r_h-\sigma_h\|_{H^{-s}(\Omega)}^2 + \|(|\sigma_h|-\mathcal{F})^+\|_{H^{-s}(\Omega)}^2 \\ &  + \langle (|\sigma_h|-\mathcal{F})^- , |u_h|\rangle - \langle \sigma_h,u_h\rangle + \langle |\sigma_h|,|u_h|\rangle.
\end{split}
\end{equation*}
\end{lem}
\begin{proof}
As in the proof of Lemma~\ref{l:Apost_obst} we obtain the following estimate:
\begin{equation}
\| u-u_h\|_{H_0^s(\Omega)}^2 + \|\sigma-\sigma_h\|^2_{H^{-s}(\Omega)} \lesssim \|r_h-\sigma_h\|_{H^{-s}(\Omega)}^2 - \langle \sigma - \sigma_h,u-u_h \rangle.
\end{equation}
In order to estimate the last term of the right hand side, we exploit the fact that $\langle \sigma,u \rangle = \langle \mathcal{F}, |u|\rangle$ and $\langle \sigma,u_h \rangle \leq \langle \mathcal{F}, |u_h|\rangle$,
\begin{equation}
\begin{split}
\langle \sigma - \sigma_h,u_h-u \rangle &\leq -\langle \mathcal{F}, |u|\rangle +\langle \sigma_h,u\rangle + \langle \mathcal{F}, |u_h|\rangle  - \langle \sigma_h,u_h\rangle \\
&\leq \langle (|\sigma_h|-\mathcal{F})^+, |u| \rangle + \langle \mathcal{F}, |u_h|\rangle  - \langle \sigma_h,u_h\rangle\\
&\leq \langle (|\sigma_h|-\mathcal{F})^+, |u-u_h| \rangle + \langle (|\sigma_h|-\mathcal{F})^+ + \mathcal{F}, |u_h|\rangle  - \langle \sigma_h,u_h\rangle\\
&\leq \|(|\sigma_h|-\mathcal{F})^+\|_{H^{-s}(\Omega)} \|u-u_h\|_{H_0^s(\Omega)} + \langle (|\sigma_h|-\mathcal{F})^- , |u_h|\rangle  \\ & \qquad- \langle \sigma_h,u_h\rangle + \langle |\sigma_h|,|u_h|\rangle.
\end{split}
\end{equation}
The result follows by combining the estimates above.
\end{proof}
Similarly, we have the following result for the friction problem:
\begin{lem}[Friction problem]
Let $u,u_h$ be solutions of Problem~\ref{p:ell2} and \ref{p:disc_ell2}, respectively, associated with Problem~\ref{prob:friction}. Assume that $f \in \mathbb{V}_h$. Then,
\begin{equation*}
\begin{split}
\|u-u_h\|_{{H}^s_{\widetilde{\Gamma}^C}(\Omega)}^2 + \|\sigma-\sigma_h\|_{H^{-s+\frac{1}{2}}(\widetilde{\Gamma})}^2 \lesssim &\|r_h-\sigma_h\|_{H^{-s}(\Omega)}^2 + \|(|\sigma_h|-\mathcal{F})^+\|_{H^{-s+\frac{1}{2}}(\widetilde{\Gamma})}^2 \\ &  + \langle (|\sigma_h|-\mathcal{F})^- , |u_h|_{\widetilde{\Gamma}}|\rangle - \langle \sigma_h,u_h|_{\widetilde{\Gamma}}\rangle + \langle |\sigma_h|,|u_h|_{\widetilde{\Gamma}}|\rangle.
\end{split}
\end{equation*}
\end{lem}

\begin{rem}
In the absence of contact, the a posteriori estimate reduces to a standard residual error estimate as in \cite{Glusa}, since $\sigma, \sigma_h$ vanish.
\end{rem}

\begin{rem}
\textcolor{black}{In line with the literature on integral equations, e.g. \cite{gwinsteph, Nochetto}, in this article we find reliable a posteriori estimates for the  error of the numerical solution. The estimates are found to be efficient in numerical experiments, but even for boundary element methods only partial theoretical results for their  efficiency are available  \cite{gwinsteph}, Chapters 10 and 12.}
\end{rem}

\subsection{Mixed formulation}
{It proves useful to impose the constraints on the displacement only indirectly. To do so, we reformulate the variational inequality as  an equivalent mixed system in which the stress $\sigma$ enters as a Lagrange multiplier. We denote it in this context by $\lambda = f - (-\Delta)^s u$ to emphasize its role. Physically, it corresponds to the contact forces and  indicates the contact area within the computational domain, see also \cite{gatica,BrezziFortin}. We focus mainly on the mixed formulations for problems with contact in the whole domain, thin problems follow in a similar way.\\

Let $a(\cdot,\cdot)$ be the bilinear form associated with the fractional Laplacian and let $b(\mu,v)$ be a continuous bilinear form given by $b(\lambda,v) = \langle f,v \rangle - a(u,v)$. Let $\Lambda$ be closed convex subset of $H^*$. For $f\in H^{*}$ and $w \in H$, we consider the mixed formulation:\\

\begin{problem}[Mixed formulation]\label{p:ell_mixed}
Find $(u,\lambda)\in H\times \Lambda$ such that 
\begin{align}\label{mixedFormulationProof}
\begin{cases}
a)~ a( u, v) + b(\lambda,v )= \langle f,v \rangle \\
b)~b(\mu- \lambda,u)\leq \langle \mu-\lambda, w \rangle,
\end{cases}
\end{align}
for all $(v,\mu)\in H \times \Lambda$.
\end{problem}
\begin{thm}\label{t:ell_equivalence}
Let $f \in H^*$, $\chi \in H^s_0(\Omega)$, $g \in H^{s-\frac{1}{2}} (\widetilde{\Gamma})$, $j_S: H^{s}_0(\Omega) \to \mathbb{R}$ and $j_{\widetilde{\Gamma}}: H^{s}_{\widetilde{\Gamma}}(\Omega) \to \mathbb{R}$  be convex lower semi-continuous functionals defined in \eqref{e:j's}. \\
Suppose that $\Lambda$ and $w$ in Problem~\ref{p:ell_mixed} is given by:
\begin{align*}
(i)\; &\Lambda_o = \lbrace \mu \in H^{-s}(\Omega):\;\forall v \in H^s_0(\Omega), \;v\leq 0,\; \langle \mu, v\rangle \geq 0 \rbrace,\; w = \chi,\\
(ii)\; &\Lambda_s = \lbrace \mu \in H^{\frac{1}{2}-s}(\widetilde{\Gamma}):\;\forall v \in H^{s-\frac{1}{2}}(\widetilde{\Gamma}), \;v\leq 0,\; \langle \mu, v\rangle \geq 0 \rbrace,\; w = g,\\
(iii)\; &\Lambda_{I} = \lbrace \mu \in H^{-s}(\Omega):\; \forall v \in H^s_0(\Omega),\; \langle \mu, v\rangle \leq \langle\mathcal{F},|v|\rangle\rbrace,\; w = 0,\\
(iv)\; &\Lambda_{\widetilde{\Gamma}} = \lbrace \mu \in H^{\frac{1}{2}-s}({\widetilde{\Gamma}}): \;\forall v \in H^{s-\frac{1}{2}}(\widetilde{\Gamma}),\; \langle \mu, v\rangle \leq \langle\mathcal{F},|v|\rangle\rbrace, \; w = 0.
\end{align*}
Then the variational inequality formulation in Problem~\ref{p:ell1}, respectively \ref{p:ell2}, is equivalent to Problem~\ref{p:ell_mixed} for Problems~\ref{prob:obst}--\ref{prob:friction}. 
\end{thm}
Let  $\Lambda_H$ be closed convex subset of $\mathbb{M}_H$. The discrete mixed formulation reads as follows:
\begin{problem}[Discrete mixed formulation]\label{p:ell_mixed_disc}
Find $(u_{h}, \lambda_{H}) \in \mathbb{V}_{h} \times \Lambda_{H}$ such that 
\begin{align}
\begin{cases} 
(a) ~ a(u_{h}, v_{h}) + b(\lambda_{H} ,v_{h})= \langle f,v_{h} \rangle \\
(b)~b(\mu_{H}- \lambda_{H},u_{h}) \leq \langle \mu_{H}-\lambda_{H}, w \rangle,
\end{cases}
\end{align}
holds for all $(v_{h},\mu_{H})\in \mathbb{V}_{h} \times \Lambda_{H}$.
\end{problem}
\begin{rem}\label{t:ell_equivalence_disc}
For problems \ref{prob:obst}--\ref{prob:friction} the corresponding $\Lambda_H$ and $w$ in Problem~\ref{p:ell_mixed_disc} are given by:
\begin{align*}
(i)\; &\Lambda_{oH} = \lbrace \mu_H \in \mathbb{M}_H: \mu_H \leq 0\rbrace,\; w = \chi,\\
(ii)\; &\Lambda_{sH} = \lbrace \mu_H \in {\widetilde{\mathbb{M}}_H}: \mu_H \leq 0\rbrace,\; w = g,\\
(iii)\; &\Lambda_{IH} = \lbrace \mu_H \in \mathbb{M}_H: |\mu_H| \leq \mathcal{F}\rbrace,\; w = 0,\\
(iv)\; &\Lambda_{\widetilde{\Gamma} H} = \lbrace \mu_H \in {\widetilde{\mathbb{M}}_H}: |\mu_H| \leq \mathcal{F}\rbrace, \; w = 0.
\end{align*}
\end{rem}

For completeness, we recall the proof of Theorem~\ref{t:ell_equivalence} for obstacle $(i)$ and interior friction $(iii)$ problems. The proof for Signorini $(ii)$ and friction $(iv)$ problems is similar.
\begin{proof}[Proof of (i)]
First note that the variational inequality is equivalent to:  Find $u\in K$ such that for all $v\in K$
\begin{align}
\label{contVIequiv}
 \begin{cases} 
a) ~ a(u,u )= \langle f,u \rangle \\
b) ~ a(u, v) \geq \langle f, v \rangle\ .
 \end{cases}
\end{align}
To see a), we set $v=2u$, respectively $v=0$, in the variational inequality: 
\begin{align}\label{e:updown}
 a(u,u ) \geq  \langle f,u \rangle\ ,\; \text{resp.}\; a(u,u) \leq \langle f,u \rangle \ ,
\end{align}
which shows a). Part b) follows by adding \eqref{e:updown} to the variational inequality. Conversely, the variational inequality follows by subtracting  (\ref{contVIequiv}a) from (\ref{contVIequiv}b). 

Further observe that $u \in H^s_0(\Omega)$ and $\lambda \in H^{-s}(\Omega)$ satisfy (\ref{mixedFormulationProof}b) if and only if 
\begin{equation}\label{eq:b_eq}
u\in K \; \mathrm{and}\; b(\lambda, u) = \langle \chi, \lambda \rangle.
\end{equation}
Indeed, \eqref{eq:b_eq} implies (\ref{mixedFormulationProof}b). Conversely, if (\ref{mixedFormulationProof}b) holds, we may choose $\mu= 0$ and $\mu = 2\lambda$ to obtain \eqref{eq:b_eq}.\\
We now show the equivalence of (\ref{contVIequiv}) and {(\ref{mixedFormulationProof})}: \\ 
 (\ref{contVIequiv}) $\Rightarrow$ {(\ref{mixedFormulationProof})}:
 If we set $\lambda= f - (-\Delta)^s u$ we have by (\ref{contVIequiv}b): $\langle f,v \rangle - a(u,v)  \leq 0$ for all $v\in H^{s}_0(\Omega)$ and therefore $\lambda \in \Lambda_o$. The first line in {(\ref{mixedFormulationProof})} holds trivially.
 \\ By (\ref{contVIequiv}a) we have that $b(\lambda,u)= \langle\lambda, \chi \rangle$ and furthermore, there exists $\hat{u} \in K$ such that $b(\mu,\hat{u}) = \langle \mu, \chi \rangle$. Also, $2u-\hat{u} \in K$ and so from \eqref{contVIequiv} we get,
 \begin{equation}
 a(u,\hat{u}-u) = \langle f , \hat{u} - u \rangle.
 \end{equation}
Substituting $v = \hat{u} - u$ into (\ref{mixedFormulationProof}a) gives us
\begin{equation}
b(\lambda,\hat{u}-u) = b(\lambda, v) - \langle \lambda, \chi \rangle = 0.
\end{equation}
As $u \in K$ and by \eqref{eq:b_eq} we conclude that (\ref{mixedFormulationProof}b) holds.\\
  {(\ref{mixedFormulationProof})} $\Rightarrow$ (\ref{contVIequiv}): 
 Now let $(u,\lambda) \in  H^{s}_0(\Omega)\times \Lambda_o$ be the solution to {(\ref{mixedFormulationProof})}. By \eqref{eq:b_eq}, we know that $u \in K$. Furthermore, by (\ref{mixedFormulationProof}a) and \eqref{eq:b_eq} we have
 \begin{equation}
 a(u,v-u) = \langle f, v-u \rangle -b(\lambda, v-u) = \langle f, v-u \rangle + \langle \lambda, \chi \rangle - b(\lambda,v) \geq \langle f, v-u \rangle.
 \end{equation}
\end{proof}

\begin{proof}[Proof of (iii)]
We begin by showing that \eqref{p:ell2} $\Rightarrow$ \eqref{p:ell_mixed}. From the variational formulation in Problem~\ref{p:ell2} we observe that we seek $u\in H^s_0(\Omega)$ such that
\begin{align}\label{e:a_and_j}
a(u,u) + j_S(u) = \langle f , u \rangle, \\
a(u,v) + j_S(v) \geq \langle f , v \rangle,
\end{align}
for all $v \in H^s_0(\Omega)$. We define $\mu \in \Lambda_I$ to be a Lagrange multiplier given by $j_S(v) = \sup_{\mu \in \Lambda_I} b(\mu,v)$.

The first line of \eqref{p:ell_mixed} hold immediately. In order to show that (\ref{mixedFormulationProof}b) holds we notice that combining (\ref{mixedFormulationProof}a) with \eqref{e:a_and_j} gives
\begin{equation}
j_S(u) = b(\lambda,u).
\end{equation}
\eqref{p:ell_mixed} $\Rightarrow$ \eqref{p:ell2} \\
Now let $(u,\lambda) \in H^s_0(\Omega) \times \Lambda_I$ be the solution to \eqref{p:ell_mixed}. By (\ref{p:ell_mixed}a) we know that
\begin{equation}
\begin{split}
a(u,v-u) &= \langle f,v-u\rangle -b(\lambda,v-u) \\ &= \langle f,v-u\rangle -b(\lambda,v) + b(\lambda,u) \\ &= \langle f,v-u\rangle -b(\lambda,v) +j_S(u) \\ &\geq \langle f,v-u\rangle -j_S(v) + j_S(u),
\end{split}
\end{equation}
where we used the definition of $j(v)$. The result follows.
\end{proof}

Note that we allow for possibly different meshes for the displacement $u$ and the Lagrange multiplier $\lambda$. 

As typical for mixed problems, this is crucial in order to assure the discrete inf-sup condition: 
\begin{lem}[Discrete inf-sup condition] There exist constants $C,\hat{\beta} >0$ such that for $H\geq Ch$
\begin{equation}\label{e:InfSupD}
\hat{\beta} \|\mu_H \|_{H^*} \leq \sup_{v_h \in \mathbb{V}_h} \frac{b(\mu_H, v_h)}{\|v_h\|_{H}}, \; \forall\; \mu_H \in \mathbb{M}_H.
\end{equation}
\end{lem}
In practice, for our choice of  $\mathbb{V}_h$, $\mathbb{M}_H$ a constant $C=2$ is sufficient. See, for example, \cite{HaslingerLovisek} for details.

\subsection{A priori error estimates for mixed formulations}
In this section, we present a priori error estimates for the elliptic problem with contact in the domain; results for thin problems follow almost verbatim.
\begin{lem}\label{l:prelim_apriori_elliptic}
Let $(u,\lambda),\;(u_h,\lambda_H)$ be solutions of Problems~\ref{p:ell_mixed} and \ref{p:ell_mixed_disc}, respectively. Suppose that $\Lambda_H \subset \Lambda$. Then
\begin{equation}
\begin{split}
\|u-u_h\|_{H}^2 \lesssim &\|u-v_h\|_{H}^2+\|\lambda-\mu_H\|_{H^*}^2 + \|\lambda-\lambda_H\|_{H^*}^2\\ & +b(\lambda-\mu_H,u) - \langle \lambda - \mu_H,w \rangle,
\end{split}
\end{equation}
for all $(v_h,\mu_H) \in \mathbb{V}_h \times \Lambda_H$.
\end{lem}
\begin{proof}
Using the coercivity of the bilinear form $a(\cdot,\cdot)$,  (\ref{mixedFormulationProof}a) and (\ref{p:ell_mixed_disc}a),
\begin{equation*}
\begin{split}
\alpha \|u-u_h\|_{H}^2 &\leq a(u-u_h,u-u_h)  = a(u,u)+a(u_h,u_h) - a(u,u_h) - a(u_h,u)\\ &\leq a(u,v) + b(\lambda,v) - b(\mu, u) +\langle f, u-v\rangle +\langle \mu-\lambda,w\rangle + \langle \mu_h-\lambda_h,w\rangle\\ &\quad+ a(u_h,v_h) + b(\lambda_H,v_h) - b(\mu_H, u_h) +\langle f, u_h-v_h\rangle  - 2a(u,u_h) \\
\iffalse
&= a(u,v-u_h)+b(\lambda,v-u_h) -\langle f,v-u_h\rangle -b(\mu - \lambda_H,u) \\ &\quad + a(u,v_h-u) + b(\lambda,v_h-u)  - \langle f, v_h-u\rangle -b(\mu_H-\lambda,u) \\ & \quad + a(u_h-u,v_h-u) +b(\lambda_H-\lambda,v_h-u) - b(\mu_H-\lambda,u_h-u) \\ & \quad+\langle \mu_H-\lambda,w\rangle + \langle \mu-\lambda_H,w\rangle\\\fi & = a(u_h-u,v_h-u) +b(\lambda_H-\lambda,v_h-u) - b(\mu_H-\lambda,u_h-u) \\ & \quad -b(\mu - \lambda_H,u) -b(\mu_H-\lambda,u) +\langle \mu_H-\lambda,w\rangle + \langle \mu-\lambda_H,w\rangle.
\end{split}
\end{equation*}
By boundedness of $a$ and $b$ and using Young's inequality
\begin{equation*}
\begin{split}
\alpha \|u-u_h\|_{H}^2 &\leq C_1 \varepsilon \|u-u_h\|_{H}^2 + C_1/\varepsilon \|u-v_h\|_{H}^2 + C_2/\varepsilon \|u-v_h\|_{H}^2\\ & +C_2 \varepsilon \|\lambda-\lambda_H\|_{H^{*}}^2 + C_2 \varepsilon \|u-u_h\|_{H}^2 + C_2/\varepsilon \|\lambda-\mu_H\|_{H^*}^2\\ & -b(\mu - \lambda_H,u) -b(\mu_H-\lambda,u) +\langle \mu_H-\lambda,w\rangle + \langle \mu-\lambda_H,w\rangle.
\end{split}
\end{equation*}
Choosing $\varepsilon>0$ sufficiently small, the result follows by combining the terms $\|u-u_h\|_{H}^2$.
\end{proof}

\begin{thm}
Let $(u,\lambda),\;(u_h,\lambda_H)$ be solutions of Problems~\ref{p:ell_mixed} and \ref{p:ell_mixed_disc}, respectively. Suppose that $\Lambda_H \subset \Lambda$. Then
\begin{align*}
\begin{split}
\|u-u_h\|_{H}^2 \lesssim &\|u-v_h\|_{H}^2+\|\lambda-\mu_H\|_{H^*}^2 + \|\lambda-\lambda_H\|_{H^*}^2,
\end{split}\\
\|\lambda-\lambda_H\|_{H^*} &\lesssim \|u-u_h\|_{H} + \|\lambda-\mu_H\|_{H^*},
\end{align*}
for all $(v_h, \mu_H) \in \mathbb{V}_h \times \Lambda_H$.
\end{thm}
The theorem adapts the classical error analysis for mixed formulations of variational inequalities for second-order elliptic operators \cite{Brezzi1977, Brezzi1978}. Optimal convergence rates depend on the regularity of solutions for the different variational inequalities. This regularity analysis is well understood for obstacle problems, see Remark \ref{regremark}.
\begin{proof}
For the first estimate we use Lemma~\ref{l:prelim_apriori_elliptic}.
For the second part, we use the discrete inf-sup condition 
\begin{equation*}
\hat{\beta} \|\lambda_H-\mu_H\|_{H^*} \leq \sup_{v_h\in \mathbb{V}_h} \frac{b(\mu_H-\lambda_H,v_h)}{\|v_h\|_{H}}.
\end{equation*}
On the other hand, from  (\ref{mixedFormulationProof}a) and (\ref{p:ell_mixed_disc}a)
\begin{equation*}
\begin{split}
b(\mu_H-\lambda_H,v_h) &= b(\mu_H,v_h) - b(\lambda_H,v_h) \\ &= b(\mu_H,v_h) +a(u_h,v_h) -\langle f, v_h \rangle \\ &=  b(\mu_H,v_h) +a(u_h,v_h) -a(u,v_h)-b(\lambda,v_h) \\ &= b(\mu_H-\lambda,v_h) + a(u_h-u,v_h) \\ &\leq c \left( \|\lambda-\mu_H\|_{H^*} +\|u-u_h\|_{H}\right)\|v_h\|_{H}.
\end{split}
\end{equation*}
Together with the inf-sup condition we conclude
\begin{equation*}
\|\lambda_H-\mu_H\|_{H^*} \lesssim \|\lambda-\mu_H\|_{H^*} +\|u-u_h\|_{H}.
\end{equation*}
The assertion follows from the triangle inequality.
\end{proof}

\subsection{A posteriori error estimates for mixed formulations}
In this section, we present a unified approach to derive a posteriori error estimates for elliptic contact problems. The contact condition only enters in the estimate for $b(\lambda_H-\lambda,u-u_h)$ below.
\begin{thm}\label{th:APostEll}
Let $(u,\lambda),\;(u_h,\lambda_H)$ be solutions of Problems~\ref{p:ell_mixed} and \ref{p:ell_mixed_disc}, respectively. Then
\begin{equation*}\label{e:Ell_apost}
\|u-u_h\|_{H}^2 + \|\lambda-\lambda_H\|_{H^*}^2 \lesssim \|r_h-\lambda_H\|^2_{H^*} + b(\lambda_H-\lambda,u-u_h).
\end{equation*}
\end{thm}

\begin{proof}
From the coercivity and the definitions of $\lambda$, respectively $\lambda_H$
\begin{equation*}
\begin{split}
\|u-u_h\|_{H_0^s(\Omega)}^2 &\lesssim a(u-u_h,u-u_h) = a(u-u_h,u-v_h) + a(u-u_h,v_h-u_h)\\
& = a(u-u_h,u-v_h) - b(\lambda-\lambda_H,v_h-u_h)\\
& = \langle f, u-v_h\rangle - a(u_h,u-v_h) - b(\lambda,u-v_h) - b(\lambda-\lambda_H,v_h-u_h) \\
& = \langle r_h, u-v_h \rangle - b(\lambda_H,u-v_h) + b(\lambda_H-\lambda,u-u_h)\\
&\lesssim \|r_h-\lambda_H\|_{H^*}\|u-v_h\|_{H} + b(\lambda_H-\lambda,u-u_h).
\end{split}
\end{equation*}
The estimate for $u$ follows from Young's inequality. For $\lambda$, we note
\begin{equation*}
\begin{split}
b(\lambda-\lambda_H,v) &= b(\lambda-\lambda_H,v-v_h) - a(u-u_h,v_h)\\
& = \langle f, v-v_h\rangle -a(u_h,v-v_h) -b(\lambda_H,v-v_h)+a(u_h-u,v)\\
& = \langle r_h, v-v_h \rangle - b(\lambda_H,v-v_h)+a(u_h-u,v)\\
& \lesssim \|r_h-\lambda_H\|_{H^*} \|v-v_h\|_{H} + \|u-u_h\|_{H}\|v\|_{H}.
\end{split}
\end{equation*}
for all $v_h$.

Choosing $v_h = 0$ we obtain
\begin{equation*}
b(\lambda-\lambda_H,v) \lesssim (\|r_h - \lambda_H\|_{H^*} + \|u-u_h\|_{H})\|v\|_{H}.
\end{equation*}
The assertion follows from the inf-sup condition.
\end{proof}

The following lemma provides computable estimates for the term $b(\lambda_H-\lambda,u-u_h)$ in case of Problems~\ref{prob:obst}--\ref{prob:friction}. 

\begin{lem}\label{l:b}
Let $(u,\lambda),\;(u_h,\lambda_H)$ be solutions of Problems~\ref{p:ell_mixed} and \ref{p:ell_mixed_disc}, respectively, associated with Problems~\ref{prob:obst}--\ref{prob:friction}. Suppose that $\Lambda_H \subset \Lambda$. Then, for the respective problems,
\begin{align*}
(i)\;b(\lambda_H-\lambda,u-u_h) \leq  & b(\lambda_H, \chi- u_h),\\
(ii)\;b(\lambda_H-\lambda,u-u_h) \leq & b(\lambda_H, g - u_h),\\
\begin{split}
(iii)\;b(\lambda_H - \lambda, u-u_h) \leq &\| (|\lambda_H| - \mathcal{F})^+ \|_{H^{-s}(S)} \|u-u_h\|_{H^s_0(S)} - b((|\lambda_H|-\mathcal{F})^-, |u_h|)\\ &\quad+ b(|\lambda_H|, |u_h|) - b(\lambda_H,u_h),
\end{split}\\
\begin{split}
(iv)\;b(\lambda_H - \lambda, u-u_h) \leq &\| (|\lambda_H| - \mathcal{F})^+ \|_{H^{1/2-s}(\widetilde{\Gamma})} \|u-u_h\|_{H^{s-1/2}(\widetilde{\Gamma})} - b((|\lambda_H|-\mathcal{F})^-, |u_h|)\\ &\quad+ b(|\lambda_H|, |u_h|) - b(\lambda_H,u_h).
\end{split}
\end{align*}
for Problems~\ref{prob:obst}--\ref{prob:friction}, respectively.
\end{lem}

\begin{proof}
$(i)$ In the case of the obstacle problem, we use the fact that $b(\lambda, u-u_h) \geq 0$ and the constraint $u-\chi \geq 0$ to obtain
\begin{equation*}\label{e:b_obstacle}
\begin{split}
b(\lambda_H-\lambda,u-u_h) &= b(\lambda_H,u-u_h)- b(\lambda,u-u_h)\\
& \leq b(\lambda_H,u - \chi) - b(\lambda_H,u_h - \chi )\\
& \leq b(\lambda_H, \chi - u_h).
\end{split}
\end{equation*}
$(ii)$ In the case of the Signorini problem, we notice that $b(\lambda, u-g) = 0$ and $b(\lambda_H, u-g) \leq 0$. The estimate follows directly as in the case of the obstacle problem \ref{pr:obst}. \\
$(iii)$ In the case of the interior friction, we notice that $b(\lambda, u) = j(u)$ and $b(\lambda,u_h) \leq j(u_h)$ to obtain the computable estimate
\begin{equation*}\label{e:b_int_friction}
\begin{split}
b(\lambda_H-\lambda,u-u_h) &= b(\lambda_H,u)-b(\lambda_H,u_h) -b(\lambda,u_h) + b(\lambda,u) \\
&\leq b(\lambda_H,u)-b(\lambda_H,u_h) +j(u_h) - j(u) \\
&\leq j(u_h) -b(\lambda_H,u_h) +b((|\lambda_H|-\mathcal{F})^+,|u|)\\
&\leq j(u_h) -b(\lambda_H,u_h) + b((|\lambda_H|-\mathcal{F})^+,|u-u_h|)\\ & \qquad + b((|\lambda_H|-\mathcal{F})^+,|u_h|)\\
&= b((|\lambda_H|-\mathcal{F})^+,|u-u_h|) - b(\lambda_H,u_h) + b(|\lambda_H|,|u_h|) \\ & \qquad- b((|\lambda_H|-\mathcal{F})^-,|u_h|)\\
&\leq \|(|\lambda_H|-\mathcal{F})^+\|_{H^{-s}(S)} \|u-u_h\|_{H^s_0(S)} - b(\lambda_H,u_h) + b(|\lambda_H|,|u_h|) \\ & \qquad- b((|\lambda_H|-\mathcal{F})^-,|u_h|).
\end{split}
\end{equation*}
We conclude by using Young's inequality for the first term.\\
$(iv)$ In the case of the friction problem, we proceed as in the case of interior friction. The only difference is in the estimate of the duality pairing in the last line of \eqref{e:b_int_friction}.
\end{proof}
Note that in the absence of constraints, the a posteriori error estimate reduces to a standard residual error estimate as in \cite{Glusa}. In order to be able to compute the negative Sobolev norm of order $-s$ on the right hand side of the estimate we can employ localization arguments as in \cite{Carstensen1, Nochetto}. The following result provides a computable estimate of the negative norms.
\begin{lem}
Let $R \in L^2(\Omega)$ which satisfies $\langle R, \phi_i \rangle = 0$ for all $i \in \mathcal{P}_h$.
Then
\begin{equation*}
\|R\|^2_{H^{-s}(\Omega)} \lesssim \sum _{i \in \mathcal{P}_h \setminus \mathcal{C}_h} h_i^{2s} \|R\|^2_{L^2(S_i)},
\end{equation*}
where $h_i:=\mathrm{diam}(S_i)$ and $\mathcal{C}_h$ is the contact region.
\end{lem}
This estimate goes back to \cite[Theorem~4.1.]{Carstensen1} in the absence of contact. For contact problems such estimates are commonly used away from the contact area, see for example \cite{BanzHeiko}. \\

Note, however, that for sufficiently large $0<s<1$ the residue $r_h-\lambda_H$ does not lie in $L^2(\Omega)$, but only in $H^{-\varepsilon}(\Omega)$ for some $\varepsilon>0$. \cite{Nochetto} extends the above arguments to $r_h - \lambda_H \in L^p$, with $1\leq p < \infty$.  For In our setting, this leads to an a posteriori error estimate as in Theorem~\ref{th:APostEll}:
\begin{equation*}
\begin{split}
\|u-u_h\|_{H^s_0(\Omega)}^2 + \|\lambda-\lambda_H\|_{H^{-s}(\Omega)}^2 \lesssim \sum_{i\in \mathcal{P}_h \setminus \mathcal{C}_h }&h_i^{2s+ d(1-\frac{2}{p})} \|(r_h-\lambda_H-r_{hi}+\lambda_{Hi})\phi_i\|^2_{L^p(S_i)} \\ &+ b(\lambda_H-\lambda,u-u_h),
\end{split}
\end{equation*}
where $g_i = \frac{\int_{S_i}g_i\phi_i}{\int_{S_i}\phi_i}$ for the interior nodes $ i \in \mathcal{P}_h$ and $g_i=0$ otherwise.

\begin{rem} \textcolor{black}{The implicit constants in the error estimates depend on $s$ through the continuity and coercivity of the bilinear form, the trace theorem for Sobolev spaces, and  properties of the triangulation as in \cite{Carstensen1}. In particular, they remain bounded for $s \to 1^{-}$.} 
\end{rem}

\section{Parabolic problems}\label{sec:Parabolic}
In this section we discuss the time-dependent counterparts to the elliptic variational inequalities. The time dependence introduces additional difficulties in the analysis. 
\subsection{A priori error estimate for variational inequalities}
We begin by the extension of Falk's lemma for parabolic Problems \ref{p:para1} and \ref{p:para2}. We present the proof only for Problem~\ref{p:para2}. The proof for Problem~\ref{p:para1} holds verbatim, omitting terms related to $j(\cdot)$.
\begin{lem}\label{thm:Para_Falk1}
Let $u\in \mathbb{W}_K(0,T)$ and $u_{h\tau} \in K_{h\tau}$ be solutions of Problem~\ref{p:para1} and \ref{p:disc_para1}, respectively. Let $v \in \mathbb{W}_K(0,T)$ and $v_{h\tau} \in K_{h\tau} \cap C(0,T;H)$. Then,
\begin{equation*}
\begin{split}
\|u-u_{h\tau}&\|_{L^2(0,T;H)}^2 + \|(u-u_{h\tau})^0_+\|_{L^2(\Omega)}^2 + \sum_{k=1}^{M-1} \|[u-u_{h\tau}]^k\|_{L^2(\Omega)}^2 +\|(u-u_{h\tau})^M_-\|_{L^2(\Omega)}^2 \\  &\lesssim  \inf_{v \in K} \lbrace \|f-Au-\partial_t u\|_{L^2(0,T;H^*)} \|u_{h\tau}-v\|_{L^2(0,T;H^s)} \rbrace \\ &+ \inf_{v_{h\tau} \in K_{h\tau}} \lbrace \|(u-v_{h\tau})^M_-\|_{L^2(\Omega)}^2 + \|f-Au-\partial_t u\|_{L^2(0,T;H^*)} \|u-v_{h\tau}\|_{L^2(0,T;H)} \\ &+ \|\partial_t(u-v_{h\tau})\|_{L^2(0,T;H^*)}^2+\|u-v_{h\tau}\|_{L^2(0,T;H)}^2  \rbrace.
\end{split}
\end{equation*}
\end{lem}

\begin{lem}\label{thm:Para_Falk}
Let $u\in \mathbb{W}(0,T)$ and $u_{h\tau} \in \mathbb{W}_{h\tau}(0,T)$ be solutions of Problem~\ref{p:para2} and \ref{p:disc_para2}, respectively. Let $v \in \mathbb{W}(0,T)$ and $v_{h\tau} \in \mathbb{W}_{h\tau}(0,T) \cap C(0,T;H)$. Then,
\begin{equation*}
\begin{split}
\|u-u_{h\tau}&\|_{L^2(0,T;H)}^2 + \|(u-u_{h\tau})^0_+\|_{L^2(\Omega)}^2 + \sum_{k=1}^{M-1} \|[u-u_{h\tau}]^k\|_{L^2(\Omega)}^2 +\|(u-u_{h\tau})^M_-\|_{L^2(\Omega)}^2 \\  &\lesssim  \inf_{v \in K} \lbrace \|f-Au-\partial_t u\|_{L^2(0,T;H^*)} \|u_{h\tau}-v\|_{L^2(0,T;H)} + \sum_{k=1}^M\int_{I_k}j(u_{h\tau})-j(v) \dt \rbrace \\ &+ \inf_{v_{h\tau} \in K_{h\tau}} \lbrace \|(u-v_{h\tau})^M_-\|_{L^2(\Omega)}^2 + \|f-Au-\partial_t u\|_{L^2(0,T;H^*)} \|u-v_{h\tau}\|_{L^2(0,T;H)} \\ &+ \|\partial_t(u-v_{h\tau})\|_{L^2(0,T;H^*)}^2+\|u-v_{h\tau}\|_{L^2(0,T;H)}^2  + \sum_{k=1}^M\int_{I_k} j(u)-j(v_{h\tau}) \dt \rbrace.
\end{split}
\end{equation*}
\end{lem}
\begin{proof}
Adding together the continuous and discrete problems gives us
\begin{equation*}
\begin{split}
B_{DG}(u,u) +& B_{DG}(u_{h\tau},u_{h\tau}) + \sum_{k=1}^M \langle [u_{h\tau}]^{k-1}, {u_{h\tau}}^{k-1}_+\rangle \\ &\leq \int_0^T \langle f,u-v_{h\tau}\rangle + \langle f,u_{h\tau}-v\rangle \dt + B_{DG}(u,v)+B_{DG}(u_{h\tau},v_{h\tau})\\ & \qquad  + \sum_{k=1}^M \int_{I_k}  j(u)-j(v_{h\tau}) \dt +\int_{I_k}  j(u_{h\tau}) - j(v)  \dt \\ & \qquad + \sum_{k=1}^M \langle [u_{h\tau}]^{k-1}, {v_{h\tau}}^{k-1}_+\rangle +\langle u_0, (u_{h\tau}-v_{h\tau})^0_+ \rangle.
\end{split}
\end{equation*}
Subtracting the mixed terms $B_{DG}(u,u_{h\tau}) + B_{DG}(u_{h\tau},u) + \sum_{k=1}^M \langle [u_{h\tau}]^{k-1}, {u}^{k-1}_+\rangle$ and using the fact that the jump terms of the continuous problem are zero,

\begin{equation*}
\begin{split}
B_{DG}(u-u_{h\tau},u - u_{h\tau})& - \sum_{k=1}^M \langle [u-u_{h\tau}]^{k-1}, {u-u_{h\tau}}^{k-1}_+\rangle \\ &\leq \int_0^T \langle f,u-v_{h\tau}\rangle + \langle f,u_{h\tau}-v\rangle\dt - B_{DG}(u,u-v_{h\tau}) - B_{DG}(u,u_{h\tau}-v) \\&  + \sum_{k=1}^M \int_{I_k} j(u)-j(v_{h\tau}) \dt +\int_{I_k} j(u_{h\tau}) - j(v) \dt \\  &+ \sum_{k=1}^M \langle [u-u_{h\tau}]^{k-1}, {v_{h\tau}-u}^{k-1}_+\rangle + B_{DG}(u-u_{h\tau},u-v_{h\tau}).
\end{split}
\end{equation*}
Due to the coercivity of the left hand side by Lemma~\ref{l:coercive}, 
\begin{equation*}
\begin{split}
\|u-u_{h\tau} &\|^2_{L^2(0,T;H)} + \|(u-u_{h\tau})^0_+\|_{L^2(\Omega)}^2 + \sum_{k=1}^{M-1} \|[u-u_{h\tau}]^k\|_{L^2(\Omega)}^2 +\|(u-u_{h\tau})^M_-\|_{L^2(\Omega)}^2   \\ &\lesssim \|f-Au-\partial_t u\|_{L^2(0,T;H^*)} \left( \|u-v_{h\tau}\|_{L^2(0,T;H)} +\|u_{h\tau}-v\|_{L^2(0,T;H)} \right)\\ &  + \sum_{k=1}^M \int_{I_k}  j(u)-j(v_{h\tau}) \dt +\int_{I_k}  j(u_{h\tau}) - j(v) \dt \\
&+ \sum_{k=1}^M \int_{I_k} \langle -u + u_{h\tau}, \partial_t(u-v_{h\tau})\rangle + a(u-u_{h\tau},u-v_{h\tau}) \dt - \sum_{k=1}^M \langle (u-u_{h\tau})^k_-,[u-v_{h\tau}]^k \rangle.
\end{split}
\end{equation*}
We can choose $v_{h\tau} \in C(0,T;H)$ and so
\begin{equation*}
\begin{split}
\|u-u_{h\tau} &\|^2_{L^2(0,T;H)} + \|(u-u_{h\tau})^0_+\|_{L^2(\Omega)}^2 + \sum_{k=1}^{M-1} \|[u-u_{h\tau}]^k\|_{L^2(\Omega)}^2 +\|(u-u_{h\tau})^M_-\|_{L^2(\Omega)}^2  \\ &\lesssim \|f-Au-\partial_t u\|_{L^2(0,T;H^*)} \left( \|u-v_{h\tau}\|_{L^2(0,T;H)} +\|u_{h\tau}-v\|_{L^2(0,T;H)} \right)\\ &+ \| u - u_{h\tau}\|_{L^2(0,T;H)} (\| \partial_t(u-v_{h\tau})\|_{L^2(0,T;H^*)} + \|u-v_{h\tau}\|_{L^2(0,T;H)}) \\ &  + \sum_{k=1}^M \int_{I_k} j(u)-j(v_{h\tau}) \dt +\int_{I_k} j(u_{h\tau}) - j(v) \dt,
\end{split}
\end{equation*}
and applying Cauchy-Schwarz yields the result.
\end{proof}

\begin{rem}\label{rem:63}
In order to obtain explicit convergence rates for the discrete solution we would like to know the regularity of the solutions. 
In the case of the unconstrained problem, we know that if $f = 0$ and $u_0 \in H^s_0(\Omega)$, then the solution $u \in \mathbb{W}(0,T)$ of the parabolic problem
\begin{equation}
\int_0^T\langle \partial_t u, v\rangle + a(u,v) \dt = \int_0^T\langle f,v \rangle \dt, \;\; v \in L^2(0,T;H^s_0(\Omega)),
\end{equation}
satisfies $u \in L^2(0,T; H^s_0(\Omega) \cap H^{s+\ell}(\Omega))$, where $\ell = \min\lbrace 1/2-\varepsilon, s \rbrace$.
We refer to \cite{Brezis1} details related to the classical cases. \\
The regularity theory of the variational inequalities discussed here is less developed than for the local elliptic case of the Laplacian. This in particular applies to the thin obstacle and friction problems, and even for the standard parabolic obstacle problem regularity is only understood under strong hypotheses \cite{figalli}. See \cite{Salgado, Silvestre} for regularity results in the elliptic case and \cite{dh} for the challenges of optimal a priori estimates for the classical Signorini problem.
\end{rem}

\subsection{A posteriori error estimate for variational inequalities}
In this section we discuss a posteriori error estimates for the parabolic variational inequalities in Problems~\ref{p:para1} and \ref{p:para2}. As before, we assume that $f$ as well as the constraint belongs to the finite element space.\\
Since a posteriori estimates require the precise formulation of the problem to determine a fully computable bound, we restrict ourselves to Problems~\ref{prob:obst}--\ref{prob:friction}.\\
We begin by discussing the parabolic obstacle problem. We define the Lagrange multiplier $\sigma$ for the parabolic problem in the following fashion as in the elliptic case,
\begin{equation}
\langle \sigma, v  \rangle = \langle f, v \rangle - a(u, v) - \langle \partial_t u,v \rangle  ,\; \forall v \in L^2(0,T;H).
\end{equation}
Furthermore, we set the residual $r_{h\tau}$ of the parabolic problem in a similar fashion as for the elliptic problems,
\begin{equation}
r_{h\tau} = f - (-\Delta)^s u_{h\tau} - \partial_t u_{h\tau},
\end{equation}
and we define $\sigma_{h\tau} \in \mathbb{W}_{h\tau}(0,T)$ to be the discrete counterpart of $\sigma$ given by
\begin{equation}
\langle \sigma_{h\tau}, v_{h\tau}  \rangle = \langle f, v_{h\tau} \rangle - a(u_{h\tau}, v_{h\tau}) - \langle \partial_t u_{h\tau},v_{h\tau} \rangle  ,\; \forall v_{h\tau} \in \mathbb{W}_{h\tau}(0,T).
\end{equation}
We will restrict ourselves to the discussion of the piecewise constant discretization in time. However, generalisation of the arguments follows directly. To this end we consider a piecewise linear interpolant $\tilde{u}$ in time, defined by
\begin{equation}
\tilde{u}(t) = u_+^k +\frac{t_k-t}{\tau_k}(u_-^k - u_+^k),
\end{equation}
for all $t \in (t_{k-1},t_k]$. This allows us to carry out similar analysis as in the case of elliptic variational inequalities. We avoid unnecessary repetition of the arguments here and present the estimates directly.\\
\begin{thm}[Obstacle problem]\label{thm:para_apost_vi}
Let $u, u_{h\tau}$ be solutions of Problems~\ref{p:para1} and \ref{p:disc_para1}, respectively, associated to the parabolic version of obstacle problem~\ref{prob:obst}. Assume that $f \in \mathbb{W}_{h\tau}(0,T)$ and $\chi \in \mathbb{V}_h$. 
Then the following computable abstract error estimate holds
\begin{equation}
\begin{split}
\|(u-u_{h\tau})&(T)\|^2_{L^2(\Omega)} + \int_0^T \|u-u_{h\tau}\|^2_{H_0^s(\Omega)} + \|\partial_t(u-u_{h\tau})\|^2_{H^{-s}(\Omega)} + \|\sigma-\sigma_{h\tau}\|^2_{H^{-s}(\Omega)} \dt\\ & \lesssim  \|(u-u_{h\tau})(0)\|^2_{L^2(\Omega)} + \int_0^T \|\tilde{u}_{h\tau}-u_{h\tau}\|^2_{H_0^s(\Omega)}+  \|r_{h\tau} - \sigma_{h\tau}\|_{H^{-s}(\Omega)}^2 \dt\\ &\qquad -\int_0^T \langle \sigma_{h\tau}, u_{h\tau}-\chi \rangle \dt.
\end{split}
\end{equation}
\end{thm}

\begin{thm}[Signorini problem]\label{thm:para_apost_vi_sign}
Let $u, u_{h\tau}$ be solutions of problems in Problems~\ref{p:para1} and \ref{p:disc_para1}, respectively, associated to the parabolic version of Signorini problem~\ref{prob:contact}. Assume that $f \in \mathbb{W}_{h\tau}(0,T)$ and $g \in \mathbb{V}_h$. 
Then the following computable abstract error estimate holds
\begin{equation}
\begin{split}
\|(u-u_{h\tau})&(T)\|^2_{L^2(\Omega)} + \int_0^T \|u-u_{h\tau}\|^2_{H_{\widetilde{\Gamma}^C}^s(\Omega)} + \|\partial_t(u-u_{h\tau})\|^2_{H^{-s}(\Omega)} + \|\sigma-\sigma_{h\tau}\|^2_{H^{-s+\frac{1}{2}}(\widetilde{\Gamma})} \dt\\ & \lesssim  \|(u-u_{h\tau})(0)\|^2_{L^2(\Omega)} + \int_0^T \|\tilde{u}_{h\tau}-u_{h\tau}\|^2_{H_{\widetilde{\Gamma}^C}^s(\Omega)}+  \|r_{h\tau} - \sigma_{h\tau}\|_{H^{-s}(\Omega)}^2 \dt\\ &\qquad -\int_0^T \langle \sigma_{h\tau}, u_{h\tau}-g \rangle \dt.
\end{split}
\end{equation}
\end{thm}

\begin{thm}[Interior friction problem]\label{thm:para_apost_vi_int_fric}
Let $u, u_{h\tau}$ be solutions of Problems~\ref{p:para2} and \ref{p:disc_para2}, respectively, associated to the parabolic version of interior friction problem~\ref{prob:int_friction}. Assume that $f \in \mathbb{W}_{h\tau}(0,T)$. 
Then the following computable abstract error estimate holds
\begin{equation}
\begin{split}
\|(u-u_{h\tau})&(T)\|^2_{L^2(\Omega)} + \int_0^T \|u-u_{h\tau}\|^2_{H_0^s(\Omega)} + \|\partial_t(u-u_{h\tau})\|^2_{H^{-s}(\Omega)} + \|\sigma-\sigma_{h\tau}\|^2_{H^{-s}(\Omega)} \dt\\ & \lesssim  \|(u-u_{h\tau})(0)\|^2_{L^2(\Omega)} + \int_0^T \|\tilde{u}_{h\tau}-u_{h\tau}\|^2_{H_0^s(\Omega)}+  \|r_{h\tau} - \sigma_{h\tau}\|_{H^{-s}(\Omega)}^2 \dt\\ &\qquad +\int_0^T \|(|\sigma_h|-\mathcal{F})^+\|_{H^{-s}(\Omega)}^2 + \langle (|\sigma_h|-\mathcal{F})^- , |u_h|\rangle - \langle \sigma_h,u_h\rangle + \langle |\sigma_h|,|u_h|\rangle \dt.
\end{split}
\end{equation}
\end{thm}

\begin{thm}[Friction problem]\label{thm:para_apost_vi_fric}
Let $u, u_{h\tau}$ be solutions of Problems~\ref{p:para2} and \ref{p:disc_para2}, respectively, asssociated to the parabolic version of friction problem~\ref{prob:friction}. Assume that $f \in \mathbb{W}_{h\tau}(0,T)$. 
Then the following computable abstract error estimate holds
\begin{equation}
\begin{split}
\|(u-u_{h\tau})&(T)\|^2_{L^2(\Omega)} + \int_0^T \|u-u_{h\tau}\|^2_{H_{\widetilde{\Gamma}^C}^s(\Omega)} + \|\partial_t(u-u_{h\tau})\|^2_{H^{-s}(\Omega)} + \|\sigma-\sigma_{h\tau}\|^2_{H^{-s+\frac{1}{2}}(\widetilde{\Gamma})} \dt\\ & \lesssim  \|(u-u_{h\tau})(0)\|^2_{L^2(\Omega)} + \int_0^T \|\tilde{u}_{h\tau}-u_{h\tau}\|^2_{H_{\widetilde{\Gamma}^C}^s(\Omega)}+  \|r_{h\tau} - \sigma_{h\tau}\|_{H^{-s}(\Omega)}^2 \dt\\ &\qquad +\int_0^T \|(|\sigma_h|-\mathcal{F})^+\|_{H^{-s+\frac{1}{2}}(\widetilde{\Gamma})}^2 + \langle (|\sigma_h|-\mathcal{F})^- , |u_h|_{\widetilde{\Gamma}}|\rangle - \langle \sigma_h,u_h|_{\widetilde{\Gamma}}\rangle + \langle |\sigma_h|,|u_h|_{\widetilde{\Gamma}}|\rangle \dt.
\end{split}
\end{equation}
\end{thm}

\subsection{Mixed formulation of the parabolic problems}
Similarly as for the elliptic problem, it proves to be useful to impose the constraint condition indirectly. Thus we reformulate the variational inequality into a mixed formulation. The Lagrange multipliers $\lambda$ provide a measure to what extend is the equality violated. Note that both Problems~\ref{p:para1} and \ref{p:para2} are covered by this reformulation. We present results for parabolic version of Problems~\ref{prob:obst}--\ref{prob:friction}.\\

Let $f\in H^*$ and $w \in H$. Let $a(\cdot,\cdot)$ be the bilinear form associated with the fractional Laplacian and let $b(\mu,v)$ be a continuous bilinear form defined analogously as for the elliptic problems. We define the continuous and discrete mixed formulation in the following way:\\
\begin{problem}\label{p:para_mixed}
Find $(u,\lambda) \in \mathbb{W}(0,T) \times \widetilde{\Lambda}$ such that,
\begin{align}
\begin{cases}
a)~\int_0^T \langle \partial_t u, v \rangle + a(u,v) + b(\lambda, v) \mathrm{d}t &= \int_0^T \langle f, v \rangle \mathrm{d}t, \\
b)~\int_0^T b(\mu - \lambda, u) \mathrm{d}t &\leq \int_0^T \langle \mu-\lambda, w \rangle \mathrm{d}t, 
\end{cases}
\end{align}
for all $v \in \mathbb{W}(0,T)$ and $\mu \in \widetilde{\Lambda}$.
\end{problem}
\begin{thm}
Let $f \in H^{*}$, $\chi \in H^s_0(\Omega)$, $g \in H^{s-\frac{1}{2}} (\widetilde{\Gamma})$, $j_S: H^{s}_0(\Omega) \to \mathbb{R}$ and $j_{\widetilde{\Gamma}}: H^{s}_{\widetilde{\Gamma}}(\Omega) \to \mathbb{R}$  be a convex lower semi-continuous functionals defined in \eqref{e:j's}. \\
Suppose that $\widetilde{\Lambda}$ and $w$ in Problem~\ref{p:para_mixed} is given by:
\begin{align*}
(i)\; &\widetilde{\Lambda}_o = \lbrace \mu \in \mathbb{W}^*(0,T): \forall v \in \mathbb{W}(0,T),\; v \leq 0,\; \langle \mu,v \rangle \geq 0\; \mathrm{a.e.}\; t \in (0,T) \rbrace, \; w = \chi,\\
(ii)\; &\widetilde{\Lambda}_s = \lbrace \mu \in L^2(0,T; H^{\frac{1}{2}-s}(\widetilde{\Gamma})): \forall v \in L^2(0,T;H^{s-\frac{1}{2}}(\widetilde{\Gamma})),\; v \leq 0,\; \langle \mu,v \rangle \geq 0,\; \mathrm{a.e.}\; t \in (0,T) \rbrace, \; w = g, \\
(iii)\; &\widetilde{\Lambda}_I = \lbrace \mu \in \mathbb{W}^*(0,T): \forall v \in \mathbb{W}(0,T),\; \langle \mu, v \rangle \leq \langle \mathcal{F}, |v|\rangle \; \mathrm{a.e.}\; t \in (0,T) \rbrace,\; w = 0,\\
(iv)\; &\widetilde{\Lambda}_{\widetilde{\Gamma}} = \lbrace \mu \in L^2(0,T;H^{\frac{1}{2}-s}(\widetilde{\Gamma})): \forall v \in L^2(0,T;H^{s-\frac{1}{2}}(\widetilde{\Gamma})),\; \langle \mu, v \rangle \leq \langle \mathcal{F}, |v|\rangle,\; \mathrm{a.e.}\; t \in (0,T) \rbrace,\; w = 0.
\end{align*}
Then the variational inequality formulation in Problem~\ref{p:para1}, respectively \ref{p:para2} is equivalent to Problem~\ref{p:para_mixed}.
\end{thm}
In our case, discretization in time is done by discontinuous Galerkin of order $q=0$ as in the case for variational inequalities. However, the analysis holds for an arbitrary $q$ with minor adjustments. For extensions of $q$ to higher degree, see \cite{Saito}.\\
Find $(u_{\tau},\lambda_{\tau}) \in  \mathbb{W}_{\tau}(0,T)\times \widetilde{\Lambda}_{o\tau}$ such that:
\begin{align}
\sum_{k=1}^M \int_{I_k} \langle \partial_t u_{\tau}, v_{\tau}  \rangle + a(u_{\tau},v_{\tau}) + b(\lambda_{\tau}, v_{\tau}) \mathrm{d}t +\sum_{k=1}^M \langle \left[u_{\tau}\right]^{k-1} , v_{\tau}^{{k-1},+} \rangle &= \sum_{k=1}^M \int_{I_k} \langle f, v_{\tau} \rangle \mathrm{d}t, \label{e:dgSemi}\\
\sum_{k=1}^M\int_{I_k} b(\lambda_{\tau}-\mu_{\tau}, u_{\tau}) \mathrm{d}t &\leq \sum_{k=1}^M\int_{I_k} \langle\lambda_{\tau}-\mu_{\tau}, u_{\tau}\rangle \mathrm{d}t,
\end{align}
for all $u_{\tau} \in \mathbb{W}_{\tau}(0,T)$ and $\mu_{\tau} \in \widetilde{\Lambda}_{o\tau}$, where $\widetilde{\Lambda}_{o\tau}$ is the time discrete counterpart of piecewise constant approximations in time of $\widetilde{\Lambda}_o$.\\
Eventhough, one can consider the continuous parabolic problem pointwise, discretization in time by discontinuous elements introduces additional jump terms which imply that the semidiscrete formulation cannot be treated pointwise. However, we can focus on one time step only:
\begin{equation*}
\int_{I_k} \langle \partial_t u_{\tau}, v_{\tau}  \rangle + a(u_{\tau},v_{\tau}) + b(\lambda_{\tau}, v_{\tau}) \mathrm{d}t + \langle \left[u\right]^{k-1} , v_{\tau}^{{k-1},+} \rangle = \int_{I_k} \langle f, v_{\tau} \rangle \mathrm{d}t. 
\end{equation*}
Note that as before, we can use the definition of $B_{DG}(\cdot,\cdot)$ to and write the semi-discrete and discrete problem in the mixed formulation:
\begin{problem}\label{p:para_mixed_H_disc}
Find $(u_{h},\lambda_{H}) \in \mathbb{W}_{h}(0,T) \times \widetilde{\Lambda}_{H}$ such that,
\begin{align}
B_{DG}(u_{h},v_{h}) + \int_{0}^T b(\lambda_{H}, v_{h}) \dt &= \int_0^T \langle f , v_{h} \rangle \dt,\\
\int_{0}^T b(\lambda_{H}-\mu_{H}, u_{h}) \mathrm{d}t &\leq \int_0^T \langle \lambda_{H}-\mu_{H}, w_{h} \rangle \mathrm{dt},
\end{align}
for the fully discrete problem for all $v_{h} \in \mathbb{W}_{h}(0,T)$ and $\mu_{H} \in \widetilde{\Lambda}_{H}$.
\end{problem}
\begin{problem}[Discrete mixed formulation]\label{p:para_mixed_disc}
Find $(u_{h\tau},\lambda_{H\tau}) \in \mathbb{W}_{h\tau}(0,T) \times \widetilde{\Lambda}_{H\tau}$ such that,
\begin{align}
B_{DG}(u_{h\tau},v_{h\tau}) +\sum_{k=1}^M \int_{I_k} b(\lambda_{H\tau}, v_{h\tau}) \dt +\sum_{k=1}^M \langle \left[u_{h\tau}\right]^{k-1} , v_{h\tau}^{{k-1},+} \rangle &= \sum_{k=1}^M \int_{I_k}\langle f , v_{h\tau} \rangle \dt,\\
\sum_{k=1}^M\int_{I_k} b(\lambda_{H\tau}-\mu_{H\tau}, u_{h\tau}) \mathrm{d}t &\leq \sum_{k=1}^M\int_{I_k} \langle \lambda_{H\tau}-\mu_{H\tau}, w_{h\tau} \rangle \mathrm{dt},
\end{align}
for the fully discrete problem for all $v_{h\tau} \in \mathbb{W}_{h\tau}(0,T)$ and $\mu_{H\tau} \in \widetilde{\Lambda}_{H\tau}$.
\end{problem}
\begin{rem}\label{t:para_equivalence_disc}
For problems \ref{prob:obst}--\ref{prob:friction} the corresponding $\widetilde{\Lambda}_{H\tau}$ and $w$ in Problem~\ref{p:para_mixed_disc} are given by:
\begin{align*}
(i)\; &\widetilde{\Lambda}_{oH\tau} = \lbrace \mu_{H\tau} \in \widetilde{\mathbb{M}}_{H\tau}: \mu_{H\tau} \leq 0\;\mathrm{a.e.}\; t \in (0,T)\rbrace,\; w = \chi,\\
(ii)\; &\widetilde{\Lambda}_{sH\tau} = \lbrace \mu_{H\tau} \in \widetilde{\mathbb{M}}_{H\tau}^{\widetilde{\Gamma}}: \mu_{H\tau} \leq 0\;\mathrm{a.e.}\; t \in (0,T)\rbrace,\; w = g,\\
(iii)\; &\widetilde{\Lambda}_{IH\tau} = \lbrace \mu_{H\tau} \in \widetilde{\mathbb{M}}_{H\tau}: |\mu_{H\tau}| \leq \mathcal{F}\;\mathrm{a.e.}\; t \in (0,T)\rbrace,\; w = 0,\\
(iv)\; &\widetilde{\Lambda}_{\widetilde{\Gamma} H\tau} = \lbrace \mu_{H\tau} \in \widetilde{\mathbb{M}}_{H\tau}^{\widetilde{\Gamma}}: |\mu_{H\tau}| \leq \mathcal{F}\;\mathrm{a.e.}\; t \in (0,T)\rbrace, \; w = 0.
\end{align*}
\end{rem}

\subsection{A priori estimates for mixed formulations}
\noindent We turn our attention to the parabolic mixed problem. In order to avoid unnecessary repetition, we denote $\widetilde{\Lambda}_o,\;\widetilde{\Lambda}_s,\;\widetilde{\Lambda}_I,\;\widetilde{\Lambda}_{\widetilde{\Gamma}}$ by $\widetilde{\Lambda}$ as well as their respective semi-discrete and discrete counterparts by $\widetilde{\Lambda}_H$, $\widetilde{\Lambda}_{\tau}$, and $\widetilde{\Lambda}_{H\tau}$. The results are presented for problems with the constraint imposed in the domain only. The arguments for thin problems follow directly.\\
Note that standard DG theory applies and we can introduce the following result. See for example \cite{Thomee}.
\begin{lem}
Let $(u_h,\lambda_H),\;(u_{h\tau},\lambda_{H\tau})$ be solutions to Problems~\ref{p:para_mixed_H_disc} and \ref{p:para_mixed_disc}, respecitvely. Then
\begin{equation*}
\|u_h(T)-u^M_{h\tau}\|_{L^2(\Omega)}^2+ \int_0^T \|u_h-u_{h\tau}\|_{H}^2 \mathrm{d}t \lesssim \sum_{k=1}^M \tau_k^{2q} \int_{I_k} \|\partial_t^q u_h\|^2_{H} \mathrm{d}t + \sum_{k=1}^M \int_{I_k} \|\lambda_H-\lambda_{H\tau} \|^2_{H^*} \mathrm{d}t.
\end{equation*}
\end{lem}
\begin{proof}
Let $\tilde{u}(t,\cdot)$ be an interpolant in time of degree $q$ of $u_h(t,\cdot)$ such that
\begin{align*}
\tilde{u}(t_k,\cdot) &= u_h(t_k,\cdot),\;\forall\;k,\\
\int_{I_k} (\tilde{u}(t,\cdot)-u_h(t,\cdot))t^{\ell} \dt &=0, \; \mathrm{for}\; \ell \leq q-2.
\end{align*}
By standard arguments for all $t \in I_k$,
\begin{equation}
\|\tilde{u}(t,\cdot) - u_h(t,\cdot)\|_{L^2(\Omega)}^2 \leq C \tau_k^{2q-1} \int_{I_k} \|\partial_t^{q} u_h\|_{H}^2 \dt.
\end{equation}
Writing $u_h - u_{h\tau} = (u_h - \tilde{u})+(\tilde{u}-u_{h\tau}) = e_1+e_2$, we note that
\begin{equation}\label{e:IntP1}
\|e_1\|_{L^2(\Omega)}^2 \leq C \sum_{k=1}^M \tau_k^{2q-1} \int_{I_k} \|\partial_t^q u_h\|_{H}^2 \dt. 
\end{equation}
Therefore, we only need to establish a bound on $e_2$. By modified Galerkin orthogonality for suitable $v$,
\begin{equation}\label{e:GO}
\int_{I_k} (\partial_t e_2, v) + a(e_2,v)+b(\lambda-\lambda_{\tau},v) \dt + (\left[e_2\right]^{k-1}, v^{k-1,+}) = -\int_{I_k} (\partial_t e_1,v)+ a(e_1,v) \dt - ([e_1]^{k-1},v^{k-1,+}).
\end{equation}
Since for all $v \in \mathbb{W}_{h\tau}(0,T)$,
\begin{equation*}
\int_{I_k} (\pt e_1,v) \dt +([e_1]^{k-1}, v^{k-1,+}) = 0,
\end{equation*}
the equation \eqref{e:GO} becomes
\begin{equation*}
\int_{I_k} (\partial_t e_2, v) + a(e_2,v)+b(\lambda-\lambda_{\tau},v) \dt + (\left[e_2\right]^{k-1}, v^{k-1,+}) = -\int_{I_k} a(e_1,v) \dt.
\end{equation*}
Furthermore, note that
\begin{equation*}
\begin{split}
2\int_{I_k} (\pt e_2,e_2) \dt +2 ([e_2]^{k-1},e_2^{k-1,+}) &= \|e_2^k\|_{L^2(\Omega)}^2-\|e_2^{k-1,+}\|_{L^2(\Omega)}^2 + 2\|e_2^{k-1,+}\|_{L^2(\Omega)}^2 -2(e_2^{k-1},e_2^{k-1,+})\\ &\geq \|e_2^k\|_{L^2(\Omega)}^2-\|e_2^{k-1}\|_{L^2(\Omega)}^2.
\end{split}
\end{equation*}
Thus,
\begin{equation*}
\|e_2^k\|_{L^2(\Omega)}^2-\|e_2^{k-1}\|_{L^2(\Omega)}^2 +2 \int_{I_k} a(e_2,e_2) + b(\lambda_H - \lambda_{H\tau},e_2) \dt \leq 2 \int_{I_k} \vert a(e_1,e_2) \vert \dt,
\end{equation*}
and by iterating through the time intervals,
\begin{equation*}
\|e_2^M\|_{L^2(\Omega)}^2 + 2 \sum_{k=0}^M \int_{I_k} a(e_2,e_2) \dt \leq \|e_2^0\|_{L^2(\Omega)}^2 + 2 \sum_{k=0}^M \int_{I_k} \vert a(e_1,e_2)\vert + b(\lambda_{H\tau} - \lambda_H,e_2) \dt.
\end{equation*}
Using coercivity and continuity of bilinear forms
\begin{equation*}
\begin{split}
\|e_2^M\|_{L^2(\Omega)}^2 &+  \sum_{k=0}^M \int_{I_k} \|e_2\|_{H}^2 \dt\\ &\lesssim \|e_2^0\|_{L^2(\Omega)}^2 +  \sum_{k=0}^M \int_{I_k} \|e_1\|_{H}^2 +\|e_2\|_{H}^2 + \|\lambda_{H\tau} - \lambda_H\|_{H^*}^2 + \|e_2\|_{H}^2 \dt.
\end{split}
\end{equation*}
Thus,
\begin{equation*}
\|e_2^M\|_{L^2(\Omega)}^2 +   \int_{0}^T \|e_2\|_{L^2(\Omega)}^2 \dt \lesssim \int_0^T \|e_1\|_{H}^2 \dt + \int_0^T \|\lambda_H-\lambda_{H\tau}\|_{H^*}^2 \dt.
\end{equation*}
The conclusion follows from the triangle inequality and estimate \eqref{e:IntP1}.
\end{proof}
\begin{lem}\label{l:para_apriori}
Let $(u,\lambda)$ and $(u_{h},\lambda_H)$ be solutions of Problems~\ref{p:para_mixed} and \ref{p:para_mixed_H_disc}, respectively. Then,
\begin{equation*}
\begin{split}
\|(u-u_h)(T)\|^2_{L^2(\Omega)}&+ \|u - u_{h}\|^2_{L^2(0,T;H)} \\&\lesssim \|\partial_t (u-v_h)\|^2_{L^2(0,T;H^*)} + \|u-v_h\|^2_{L^2(0,T;H)} \\ & \;+ \|\lambda-\lambda_H\|^2_{L^2(0,T;H^*)} + \|\lambda-\mu_H\|^2_{L^2(0,T;H^*)} \\ & \;- \int_0^T b(\mu_H-\lambda,u) - \langle \mu_H - \lambda, w\rangle\dt - \int_0^T b(\mu-\lambda_H,u) - \langle \mu - \lambda_H, w\rangle \dt,
\end{split}
\end{equation*}
for all $(v_h,\mu_H) \in \mathbb{W}_h(0,T) \times \widetilde{\Lambda}_H$, where $\widetilde{\Lambda}_H$ is the corresponding semidiscrete space related to Problems~\ref{prob:obst}--\ref{prob:friction}.
\end{lem}
\begin{proof}
By coercivity of the bilinear form $B_{DG}(\cdot,\cdot)$,
\begin{equation*}
\begin{split}
\|u-u_h\|^2_{L^2(0,T;H)} &\lesssim B_{DG}(u-u_h,u-u_h) + \int_0^T b(\lambda-\lambda_H,u-u_h) \dt \\ & \quad- \int_0^T b(\lambda-\lambda_H,u-u_h) \dt \\
& = B_{DG}(u-u_h,u-v_h) +\int_0^T b(\lambda-\lambda_H,u-v_h) \dt\\ & \quad - \int_0^T b(\lambda-\lambda_H,u-u_h) \dt \\
& \leq B_{DG}(u-u_h,u-v_h) +\int_0^T b(\lambda-\lambda_H,u-v_h) \dt\\ & \quad -\int_0^T b(\mu-\lambda_H,u) \dt - \int_0^T b(\mu_H-\lambda,u_h) \dt\\ &\quad +\int_0^T \langle \mu - \lambda_H,w \rangle \;\mathrm{dt} +\int_0^T \langle \mu_H - \lambda,w \rangle\; \mathrm{dt} \\
& =  B_{DG}(u-u_h,u-v_h) +\int_0^T b(\lambda-\lambda_H,u-v_h) \dt \\ & \quad+\int_0^T b(\mu_H-\lambda,u-u_h) \dt -\int_0^T b(\mu-\lambda_H,u) \dt \\ & \quad- \int_0^T b(\mu_H-\lambda,u) \dt +\int_0^T \langle \mu - \lambda_H,w \rangle\; \mathrm{dt} +\int_0^T \langle \mu_H - \lambda,w \rangle\; \mathrm{dt},
\end{split}
\end{equation*}
where we have used the constraint on $b(\cdot,\cdot)$. Then integration by parts in time for the bilinear form $B_{DG}(\cdot,\cdot)$ yields the result.
\end{proof}
\begin{rem}\label{rem:1}
Let $U$ be a fully discrete solution of degree $q$ in time. Then there exists a piecewise polynomial function $\mathcal{U} \in \mathbb{W}(0,T)$ of degree $q+1$ such that it interpolates $U$ at the local points,
\begin{equation}
\mathcal{U}(\tau^j_k) = U(t^j_k),\; \mathrm{for\;all}\; j = 1,\dots ,q+1.
\end{equation}
Furthermore, imposing continuity gives
\begin{equation}
\mathcal{U}(t_{k-1}) = U_{k-1}^-.
\end{equation}
Thus $\mathcal{U}$ is uniquely defined as 
\begin{equation}
\mathcal{U}|_{I_k} := \sum_{j=0}^{q+1} \mathcal{L}_j\left(\frac{t-t_{k-1}}{\tau_{k}}\right) U(t_{k}^j),
\end{equation}
where $\mathcal{L}_j$ are Lagrange polynomials. Then from integration by parts in time, observe that \eqref{e:dgSemi} is equivalent to
\begin{equation}
\sum_{k}\int_{I_k} \langle \partial_t \mathcal{U}, V \rangle + a(U,V) +b(\Lambda,V) \dt = \sum_{k} \int_{I_k} \langle f,V\rangle \dt,
\end{equation}
where $V$ is a piecewise polynomial function of degree $q$ in time.
\end{rem}

\begin{lem}
Let $(u,\lambda),\;(u_{h\tau},\lambda_{H\tau})$ be solutions of Problems~\ref{p:para_mixed} and \ref{p:para_mixed_disc}, respectively. Under the assumption that the discrete inf-sup condition holds,
\begin{equation*}
\|\lambda-\lambda_{H\tau}\|_{L^2(0,T;H^{*})}^2 \lesssim \|\lambda - \mu_{H\tau}\|^2_{L^2(0,T;H^{*})} + \|u-u_{h\tau}\|^2_{L^2(0,T;H)} + \|\partial_t( u - u_{h\tau})\|^2_{L^2(0,T;H^*)},
\end{equation*}
for all $\mu_{H\tau} \in \widetilde{\Lambda}_{H\tau}$.
\end{lem}
\begin{proof}
By Remark~\ref{rem:1} we consider the bilinear form pointwise in time using interpolant $\mathcal{U}$. Then,
\begin{equation*}
\begin{split}
b(\mu_{H\tau}-\lambda_{H\tau},v_{h\tau}) &= b(\mu_{H\tau},v_{h\tau}) - b(\lambda_{H\tau},v_{h\tau})\\
& = b(\mu_{H\tau},v_{h\tau}) +\langle \partial_t \mathcal{U},v_{h\tau}\rangle +a(u_{h\tau}, v_{h\tau})- \langle f, v_{h\tau}\rangle \\
& = b(\mu_{H\tau} - \lambda, v_{h\tau} ) + \partial_t (\mathcal{U}-u),v_{h\tau}\rangle +a(u_{h\tau}-u, v_{h\tau})\\
&\lesssim \left( \|\mu_{H\tau} - \lambda\|_{H^*} + \|\partial_t (\mathcal{U}-u_{h\tau})\|_{H^*}\right. \\ & \left.\;\;+ \|\partial_t (u-u_{h\tau})\|_{H^*} + \|u-u_{h\tau}\|_{H} \right) \|v_{h\tau}\|_{H}.
\end{split}
\end{equation*}
Using standard approximation properties, the discrete inf-sup condition and integrating in time yields the desired result.
\end{proof}

\subsection{A posteriori analysis of mixed formulation}

Similarly as in the case of the elliptic mixed problem, we begin by estimating the error of the approximate and exact solution in the energy norm. We begin by pointing out an estimate for the bilinear form from \cite{Nochetto0}.
\begin{lem}\label{l:a_est}
Let $a(\cdot,\cdot)$ be a continuous and coercive bilinear form. Then,
\begin{equation*}
a(u-v,w-v) \geq \frac{\alpha}{4} \left( \|u-v\|^2 + \|w-v\|^2\right) - \frac{C}{2}\|u-w\|^2.
\end{equation*}
\end{lem}
\begin{proof}
Since $a(\cdot,\cdot)$ is coercive
\begin{equation*}
\begin{split}
0 &\leq a(u-v,u-v) + a(w-v,w-v) \\
&= a(u-v,2w-u-v) + a(w-v, 2u - v-w) + 2 a(u-w,u-w) \\
&= 4a(u-v,w-v) - a(u-v,u-v) - a(w-v,w-v) +2 a(u-w,u-w)\\
&\leq 4a(u-v,w-v) - \alpha \|u-v\|^2 - \alpha \|v-w\|^2 + 2C \|u-w\|^2
\end{split}
\end{equation*}
where we used the coercivity and continuity of the bilinear form.
\end{proof}

\begin{thm}\label{l:para_a}
Let $(u,\lambda),\;(u_{h\tau},\lambda_{H\tau})$ be solutions of Problems~\ref{p:para_mixed} and \ref{p:para_mixed_disc}, respectively. Furthermore, let $\mathcal{U}$ be interpolant of $u_{h\tau}$ defined in Remark~\ref{rem:1}. Let $r_{h\tau} = f - \partial_t \mathcal{U}- (-\Delta)^s u_{h\tau}$. Then,

\begin{equation*} 
\begin{split}
\|(u-\mathcal{U})(T)\|_{L^2(\Omega)}^2 &+ \sum_{k=1}^M \int_{I_k} \|\partial_t(u-\mathcal{U}) + (\lambda - \lambda_{H\tau})\|^2_{H^{*}} + \|u-\mathcal{U}\|^2_{H} + \|u-u_{h\tau}\|^2_{H} \dt \\ &\lesssim \|(u-\mathcal{U})(0)\|_{L^2(\Omega)}^2 + \sum_{k=1}^M \int_{I_k} b(\lambda_{H\tau}-\lambda,u-\mathcal{U}) \dt \\ &+ \sum_{k=1}^M \int_{I_k} \|r_{h\tau}-\lambda_{H\tau}\|_{H^{*}}^2 + \|u_{h\tau} - \mathcal{U}\|_{H}^2 \dt.
\end{split}
\end{equation*}

\end{thm}
\begin{proof}

We note that for all $v \in H$,
\begin{equation}\label{e:Galerkin}
\langle r_h, v \rangle - b(\lambda_{H\tau},v)  = \langle \partial_t(u-\mathcal{U},v\rangle + a (u-u_h, v) + b(\lambda - \lambda_{H\tau},v).
\end{equation}
By choosing $v = u - \mathcal{U}$ we obtain 
\begin{equation*}
\frac{1}{2} \frac{\mathrm{d}}{\dt} \|u-\mathcal{U}\|^2_{L^2(\Omega)} + a(u-u_{h\tau}, u - \mathcal{U}) = \langle r_h - \lambda_{H\tau}, u -\mathcal{U} \rangle + b(\lambda_{H\tau} - \lambda,u - \mathcal{U}).
\end{equation*}
Using the estimate from Lemma~\ref{l:a_est}, we obtain
\begin{equation*}
\begin{split}
\frac{1}{2} \frac{\mathrm{d}}{\dt} \|u-\mathcal{U}\|^2_{L^2(\Omega)} &+ \frac{\alpha}{4} \left( \|u-v\|_{H}^2 + \|w-v\|_{H}^2\right)\\ &\leq \frac{C}{2}\|u-w\|_{H}^2 + \langle r_h - \lambda_{H\tau}, u -\mathcal{U} \rangle + b(\lambda_{H\tau} - \lambda,u - \mathcal{U}).
\end{split}
\end{equation*}
Additionally, using \eqref{e:Galerkin} and continuity of the bilinear pairs gives
\begin{equation*}
\|\partial_t (u - \mathcal{U}) + (\lambda-\lambda_{H\tau}) \|_{H^{*}}^2 \lesssim \|r_h - \lambda_{H\tau}\|^2_{H^{*}} + \|u-u_{h\tau}\|^2_{H}.
\end{equation*}
Combining these estimates and integrating in time yields the result.
\end{proof}

\begin{rem}
We note that the last term is not yet computable. For a specific problem, we use different estimates dependent on the precise formulation of the problem. As an example of treatment of such term see \cite{BanzHeiko} for contact problems. Noting that the integration in time over time does not introduce any additional issues we can treat the $b(\cdot,\cdot)$ terms is similar as in the case of the elliptic problems and only then integrate in time. Therefore, in order to avoid repetition, we refer the reader to Lemma~\ref{l:b} and only state the resulting computable error estimates for the $b(\cdot,\cdot)$ term.
\end{rem}
\begin{lem}\label{l:b_para}
Let $(u,\lambda),\;(u_{h\tau},\lambda_{H\tau})$ be solutions of Problems~\ref{p:para_mixed} and \ref{p:para_mixed_disc}, respectively, associated with problem~\ref{prob:obst}--\ref{prob:friction}. Suppose that $\widetilde{\Lambda}_{H\tau} \subset \widetilde{\Lambda}$. Then, for the respective problems,
\begin{align*}
(i)\;\sum_{k=1}^M \int_{I_k} b(\lambda_{H\tau}-\lambda,u-\mathcal{U}) \dt \leq & \sum_{k=1}^M \int_{I_k}b(\lambda_{H\tau},\chi - \mathcal{U}) \dt,\\
(ii)\;\sum_{k=1}^M \int_{I_k} b(\lambda_{H\tau}-\lambda,u-\mathcal{U}) \dt \leq & \sum_{k=1}^M \int_{I_k}b(\lambda_{H\tau},g-\mathcal{U}) \dt,\\
\begin{split}
(iii)\;\sum_{k=1}^M \int_{I_k} b(\lambda_{H\tau}-\lambda,u-\mathcal{U}) \dt \leq &\| (|\lambda_{H\tau}| - \mathcal{F})^+ \|_{L^2(0,T;H^{-s}(S))} \|u-\mathcal{U}\|_{L^2(0,T;H^s_0(S))} \\ &- \sum_{k=1}^M \int_{I_k} b((|\lambda_{H\tau}|-\mathcal{F})^-, |\mathcal{U}|)+ b(|\lambda_{H\tau}|, |\mathcal{U}|) - b(\lambda_{H\tau},\mathcal{U})\dt,
\end{split}\\
\begin{split}
(iv)\;\sum_{k=1}^M \int_{I_k} b(\lambda_{H\tau}-\lambda,u-\mathcal{U}) \dt \leq &\| (|\lambda_{H\tau}| - \mathcal{F})^+ \|_{L^2(0,T;H^{1/2-s}(\widetilde{\Gamma}))} \|u-\mathcal{U}\|_{L^2(0,T;H^{s-1/2}(\widetilde{\Gamma}))} \\ &- \sum_{k=1}^M \int_{I_k} b((|\lambda_{H\tau}|-\mathcal{F})^-, |\mathcal{U}|)+ b(|\lambda_{H\tau}|, |\mathcal{U}|) - b(\lambda_{H\tau},\mathcal{U}) \dt.
\end{split}
\end{align*}
for the parabolic version of problems~\ref{prob:obst}--\ref{prob:friction}, respectively.
\end{lem}

\begin{rem}
As the discrete constraint $\lambda_{H\tau}$ is imposed on a coarser mesh, in order to simplify the implementation, it would be useful to be able to impose $\lambda_{H\tau}$ on the same mesh as the solution $u_{h\tau}$. To this end one could try to attack this problem using stabilization techniques as discussed for example in \cite{BarrenecheaChouly}.
\end{rem}

\section{Algorithmic aspects}\label{sec:implementation}
In this section we address the implementation of the bilinear form $a(\cdot,\cdot)$ associated with the fractional Laplacian, an Uzawa algorithm for the solution of the variational inequality, as well as adaptive mesh refinement procedures. \\  In the nodal basis $\lbrace \phi_i\rbrace$ of $\mathbb{V}_h$ the stiffness matrix $K = (K_{ij})$ is given by
\begin{equation*}
K_{ij} = a(\phi_i, \phi_j) =  c_{n,s} \iint_{\mathbb{R}^n \times \mathbb{R}^n} \frac{(\phi_i(x)-\phi_i(y))(\phi_j(x)-\phi_j(y))}{|x-y|^{n+2s}} \mathrm{d}y \, \mathrm{d}x\ .
\end{equation*}
Noting that interactions in $\Omega \times \Omega^C$ and $\Omega^C \times \Omega$ are symmetric,
\begin{equation*}
\begin{split}
K_{ij} = c_{n,s} \iint_{\Omega \times \Omega}& \frac{(\phi_i(x)-\phi_i(y))(\phi_j(x)-\phi_j(y))}{|x-y|^{n+2s}} \mathrm{d}y \, \mathrm{d}x \\ &+ 2 c_{n,s} \iint_{\Omega \times \Omega^C} \frac{\phi_i(x)\phi_j(x)}{|x-y|^{n+2s}} \mathrm{d}y \, \mathrm{d}x\ .
\end{split}
\end{equation*}
The first integral is computed using a composite graded quadrature as standard in boundary element methods \cite[Chapter 5]{SauterSchwabBook}. This method splits the integral into singular and regular parts. It converts the integral over two elements into an integral over $[0,1]^4$ and resolves the singular part with an appropriate grading. The singular part can be computed explicitly, and for the regular part we employ numerical quadrature. The second integral can be efficiently computed by numerical quadrature after transforming it into polar coordinates. For a discussion of the quadrature see \cite{Acosta2}.\\
In order to solve problems associated to the mixed formulation of variational inequalities, we use an Uzawa algorithm similar to \cite{HeikoTransmission}.  Let $P_{\Lambda_H}$ be the orthogonal projection onto $\Lambda_H$. In practice we choose $H = 2h$; stabilized methods with $H = h$ will be the content of future work.

\begin{alg}[Uzawa]
Inputs: Choose $\lambda_H^0  \in \Lambda_H$.
\begin{enumerate}
\item For $n\geq 0$ find $u_h^n \in \mathbb{V}_h$ such that
\begin{equation*}
a(u_h^n,v_h) + b(\lambda_H^n,v_h) = \langle f, v_h \rangle
\end{equation*}
for all $v_h \in \mathbb{V}_h$.
\item For appropriately chosen $\rho>0$ set
\begin{equation*}
\lambda_H^{n+1} = P_{\Lambda} (\lambda_H^n + \rho (B u_h^n - g))
\end{equation*}
\item Repeat 1. and 2. until convergence criterion is satisfied.
\end{enumerate}
Output: Solution $(u_h^{n+1}, \lambda_h^{n+1})$.
\end{alg}
\noindent Note that for the time-dependent problems, $f$ involves information from the previous time step.\\

\textcolor{black}{Because the bilinear form $a$ is coercive with coercivity constant $\alpha$, a standard argument shows that the Uzawa algorithm converges for $0<\rho <2 \alpha$. See, for example, Lemma 22 in \cite{HeikoContact}. The optimal choice for the parameter $\rho$ is $\frac{2}{\lambda_{max} + \lambda_{min}}$, where $\lambda_{max}, \ \lambda_{min}$ correspond to the largest, respectively smallest eigenvalues of $BA^{-1}B^{T}$, and this value for $\rho$ is used in the numerical experiments below.}\\

The adaptive algorithm follows the established sequence of steps:
\begin{equation*}
\mathrm{SOLVE}\to\mathrm{ESTIMATE}\to\mathrm{MARK}\to\mathrm{REFINE}.
\end{equation*}

The precise algorithm for time-independent problems is given as follows:
\begin{alg}[Adaptive algorithm 1]\label{Alg:Adaptive}
Inputs: Spatial meshes $\mathcal{T}_h$ and $\mathcal{T}_H$, refinement parameter $\theta \in (0,1)$, tolerance $\varepsilon >0$, data $f$.
\begin{enumerate}
\item Solve problem \ref{p:ell_mixed_disc}, for $(u_h, \lambda_H)$ on $\mathcal{T}_h \times \mathcal{T}_H$.
\item Compute error indicators $\eta(\Delta)$ in each triangle $\Delta \in \mathcal{T}$.
\item Find $\eta_{max} = \max_{\Delta} \eta(\Delta).$
\item Stop if $\sum_{i} \eta(\Delta_i) \leq \varepsilon$.
\item Mark all $\Delta$ with $\eta(\Delta_i) > \theta \eta_{max}$.
\item Refine each marked triangle to obtain new mesh $\mathcal{T}_h$.
\item Repeat until convergence criterion is satisfied.
\end{enumerate}
Output: Solution $(u_{h}, \lambda_{H})$.
\end{alg}

\textcolor{black}{
Here, we define the local error indicators $\eta(\Delta)$ for all elements $\Delta$ using the right hand side of the a posteriori estimate of Theorem~\ref{th:APostEll}. We approximate the dual norm $\| \mu_H \|_{H^{-\alpha}}$ by the scaled $L^2$-norm $H^{\alpha} \|\mu_H\|_{L^2}$ as well as $\|v_h\|_{H^{\alpha}}$ by $h^{-\alpha}\|v_h\|_{L^2}$ for $\alpha>0$, as standard for boundary element methods \cite{BanzHeiko}:
\begin{equation*}
\eta(\Delta)^2 = \sum_{\substack{S_i \cap \Delta \neq \varnothing \\ i \notin \mathcal{C}_h}} h_i^{2s} \|(r_h-\lambda_H-\bar{r}_{h}+\bar{\lambda}_{H})\phi_i\|^2_{L^2(S_i)} + b(\lambda_H-\lambda,u-u_h)\ .
\end{equation*}
Here, $\bar{r}_{h} = \frac{\int_{S_i}r_h\phi_i}{\int_{S_i}\phi_i}$ for the interior nodes $i \in \mathcal{P}_h$, and $\bar{r}_h=0$ otherwise. Similarly, $\bar{\lambda}_{H}= \frac{\int_{S_i}\lambda_H\phi_i}{\int_{S_i}\phi_i}$ for the interior nodes $i \in \mathcal{P}_h$, and $\bar{\lambda}_H=0$ otherwise. 
The bilinear form $b$ is estimated using Lemma~\ref{l:b} for the given problem. All integrals are evaluated using a numerical Gauss-Legendre quadrature.}\\

An algorithm for time-dependent problems is given by:
\begin{alg}[Adaptive algorithm 2]\label{Alg:AdaptiveP}
Inputs: Space-time meshes $\mathcal{S}_h =\mathcal{T}_h \times \bigcup_k I_k$ and $\mathcal{S}_H =\mathcal{T}_H \times \bigcup_k I_k$, refinement parameter $\theta \in (0,1)$, tolerance $\varepsilon >0$, data $f, u_0$.
\begin{enumerate}
\item Solve the problem \ref{p:para_mixed_disc}, for $(u_{h\tau}, \lambda_{H\tau})$ on $\mathcal{S}_h \times \mathcal{S}_H$.
\item Compute error indicators $\eta(\Delta)$ in each space-time prism $\Delta \in \mathcal{S}_h $.
\item Find $\eta_{max} = \max_{\Delta} \eta(\Delta).$
\item Stop if $\sum_{i} \eta(\Delta_i) \leq \varepsilon$.
\item Mark all $\Delta$ with $\eta(\Delta_i) > \theta \eta_{max}$.
\item Refine each marked space-time prism to obtain new mesh $\mathcal{S}_h$, keeping $\frac{\tau}{h^{2s}}$ fixed.
\item Repeat until convergence criterion is satisfied.
\end{enumerate}
Output: Solution $(u_{h\tau}, \lambda_{H\tau})$.
\end{alg}

\textcolor{black}{The error indicators $\eta(\Delta)$ are evaluated analogous to the time-independent case, using the right hand side of Theorem~\ref{l:para_a} and Lemma~\ref{l:b_para} to estimate the bilinear form $b$ for the given problem:
\begin{equation*}
\begin{split}
\eta(\Delta)^2 = \sum_{\substack{S_i \cap \Delta \neq \varnothing \\ i \notin \mathcal{C}_h}} \left( \sum_{k=1}^M \tau_k h_i^{2s} \|r_{h\tau}-\lambda_{H\tau} -\bar{r}_{h\tau}+\bar{\lambda}_{H\tau} \|_{L^2(S_i)}^2 + \tau_k h_i^{-2s} \|u^k_{h\tau} - u^{k-1}_{h\tau} \|_{L^2(S_i)}^2 \right. \\
\left.+\|(u-\mathcal{U})(0)\|_{L^2(S_i)}^2 + \sum_{k=1}^M \tau_k \, b(\lambda_{H\tau}-\lambda,u-\mathcal{U})\right).
\end{split}
\end{equation*}}

\begin{rem}
Special attention has to be paid to evaluation of $(-\Delta)^s u_h$, which is a part of the residual $r_h$. Pointwise values can be computed at the quadrature points of an appropriate quadrature rule, see \cite{Glusa}. Evaluation of the negative Sobolev norm in the a posteriori estimates is done by localization of the norm and extraction of powers of $h$ see \cite{Carstensen1, Nochetto}.
\end{rem}

\section{Numerical results}\label{sec:Numerics}

This section illustrates the a posteriori error estimates from Theorem~\ref{th:APostEll} and \ref{l:para_a} and shows the efficiency of the resulting adaptive mesh refinements from Algorithms~\ref{Alg:Adaptive} and \ref{Alg:AdaptiveP}.\\

\subsection{Time-independent problems}
Before doing so, we consider as a reference the fractional Laplace equation.
\begin{align}\label{e:FL}
(-\Delta)^s u &= f \;\mathrm{in}\; \Omega.\\
u &= 0 \;\mathrm{in}\; \Omega^C.
\end{align}
Its weak formulation reads:\\
Find $u \in H^s_0(\Omega)$ such that
\begin{equation}\label{e:FLweak}
a(u,v) = \int_{\Omega} f v \mathrm{d}x,
\end{equation}
for all $v \in H^s_0(\Omega)$. 

\begin{example}\label{ex:Equation}
We consider the fractional Laplace equation \eqref{e:FL} in $\Omega = B_1(0)$ with  $s=0.5$ and $f=1$. The exact solution is given by $u(x) = (1-\vert x \vert^2)^s_{+}$. We compare the solution to the Galerkin solution to \eqref{e:FLweak} by piecewise linear finite elements on uniform, graded, and adaptively refined meshes to the exact solution. Figure \ref{f:equalgra} shows the numerical solution on a $2$-graded mesh. Figure \ref{f:Convergence_eq} plots the error in the $H^s(\Omega)$ norm for the different meshes in terms of the degrees of freedom. The rate of convergence in terms of degrees of freedom is $ -0.252$ for uniform meshes, $-0.540$ for $2$-graded meshes, and $-0.510$ for adaptively generated meshes. This corresponds to a convergence rate of $0.504$ (uniform), $1.08$ ($2$-graded), respectively $1.02$ (adaptive), in terms of the mesh size $h$. For the uniform and graded meshes this is in agreement with the theoretically predicted rates of $0.5$ (uniform) and $1.0$ ($2$-graded), respectively. For the adaptive algorithm it agrees with the rates observed for integral operators in stationary and time-dependent problems \cite{Acosta1, Carstensen1, Gimperlein2018}.
\end{example}

\begin{figure}
\center
\includegraphics[width = 0.7\textwidth]{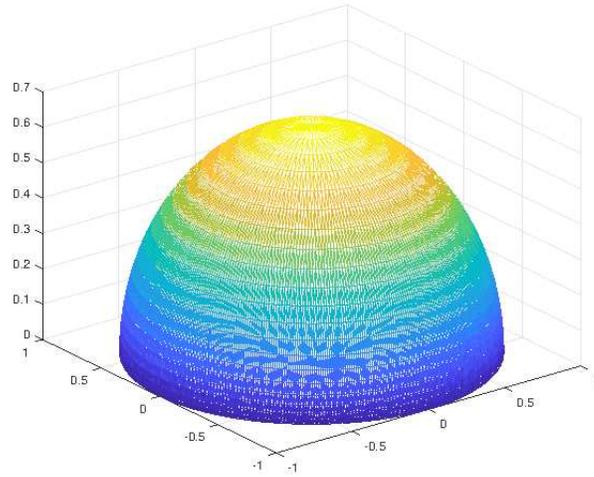}
\caption{Solution $u_h$ of the fractional Laplace equation with $s=0.5$ on graded meshes.}
\label{f:equalgra}
\end{figure}
\begin{figure}
\center
\includegraphics[width = 0.7\textwidth]{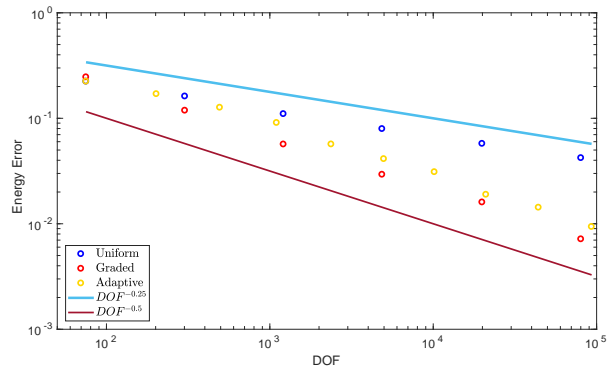}
\caption{Error in the energy norm for the fractional Laplace equation with $s=0.5$.}
\label{f:Convergence_eq}
\end{figure}

We now consider an elliptic obstacle problem:
\begin{example}\label{ex:1}
We consider the mixed formulation of the fractional obstacle Problem \ref{p:ell_mixed} in $\Omega = B_1(0)$ with  $s=0.5$, $f = 1$ and obstacle $\chi$ depicted in Figure~\ref{f:Obstacle_chi}.  We compare the Galerkin solution to \eqref{p:ell_mixed_disc} by piecewise linear finite elements on uniform, graded, and adaptively refined meshes with a benchmark solution on an adaptively generated mesh with $237182$ degrees of freedom. Figures~\ref{f:UniformObstacle} and \ref{f:AdaptiveObstacle} show the numerical solutions on a uniform and on an adaptively refined mesh, respectively. \textcolor{black}{Note that due to the strong boundary singularity of the solution the adaptive refinement is particularly strong near the boundary, as well as near the free boundary. This observation underlines the recent analysis of the obstacle problem in \cite{Salgado}.} Figure~\ref{f:Convergence_ineq_ell} shows the error in the $H^s(\Omega)$ norm for the different meshes in terms of the degrees of freedom. The error indicators capture the slope of the error in the adaptive procedure, indicating the efficiency and reliability of the a posteriori error estimates. The rate of convergence in terms of degrees of freedom is $-0.245$ for uniform meshes, $-0.498$ for $2$-graded meshes, and $-0.363$ for adaptively generated meshes. This corresponds to a convergence rate of $0.490$ (uniform), $0.996$ ($2$-graded) in terms of the mesh size $h$. We note that the graded meshes double the convergence rate of the uniform meshes, as has been recently discussed in \cite{Salgado} for the obstacle problem. Figure~\ref{f:MeshesEl} depicts the 1, 3, 8 and 15th mesh created by the adaptive algorithm. They show strong refinement near the boundaries, as well as refinement near the contact boundary.
\end{example}

\begin{rem}
Algebraically graded meshes are known to lead to quasioptimal convergence rates for the boundary singularities near the Dirichlet boundary. However, for large grading parameter their accuracy for integral equations is often limited by floating point errors, and related works consider $2$-graded meshes \cite{Acosta1}. Furthermore, graded meshes do not refine near the free boundary, which becomes relevant for the absolute size of the error as in Figures~\ref{f:ErIndVsTime} and \ref{f:para_conv}, even if not for the convergence rate.\\
On the other hand, the flexibility of adaptive meshes in complex geometries proves useful in applications. Adaptively generated meshes moreover resolve space-time inhomogeneities and singularities of solutions, as seen in Figures~\ref{f:ErIndVsTime} and \ref{f:para_conv}. Even though adaptive meshes are locally quasi-uniform, the associated convergence rates can be slower than for the anisotropic graded meshes, which involve arbitrarily thin triangles near the boundary. A heuristic explanation for the substantially higher rates of anisotropic graded meshes can be found in \cite{Carstensen1}.
\end{rem}

\begin{figure}
\centering
\includegraphics[width = .7\textwidth]{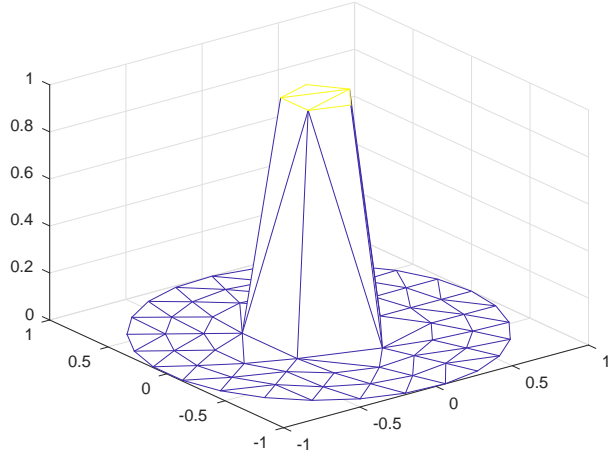}
\caption{Obstacle $\chi$ for the elliptic and parabolic obstacle problem.}
\label{f:Obstacle_chi}
\end{figure}

\begin{figure}
\center
\includegraphics[width = 0.7\textwidth]{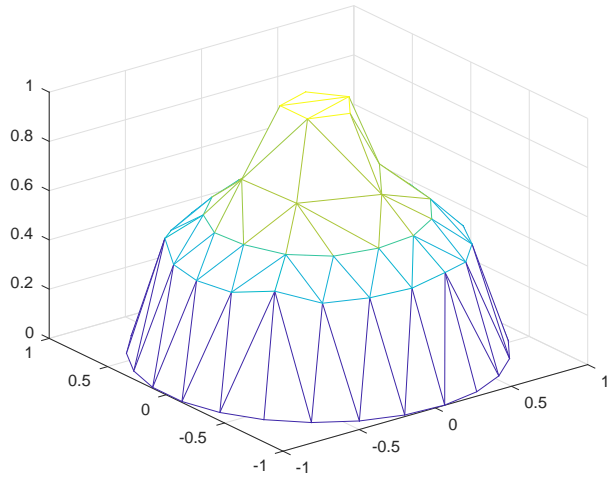}
\caption{Solution $u_h$ of Example~\ref{ex:1} on uniform mesh.}
\label{f:UniformObstacle}
\end{figure}
\begin{figure}
\center
\includegraphics[width = 0.7\textwidth]{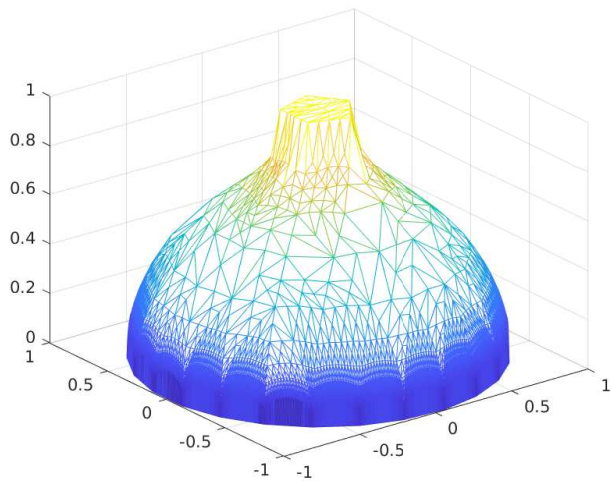}
\caption{Solution $u_h$ of Example~\ref{ex:1} using adaptive mesh refinement.}
\label{f:AdaptiveObstacle}
\end{figure}
\begin{figure}
\center
\includegraphics[width = \textwidth]{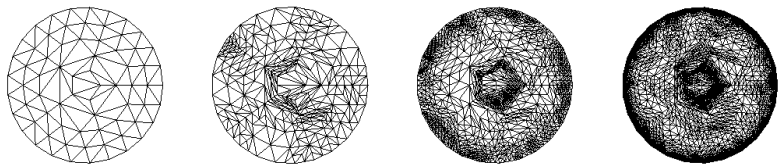}
\caption{Adaptively refined meshes for Example~\ref{ex:1} after $0,\;2,\;7,\;14$ refinements.}
\label{f:MeshesEl}
\end{figure}

\begin{figure}
\center
\includegraphics[width = 0.7\textwidth]{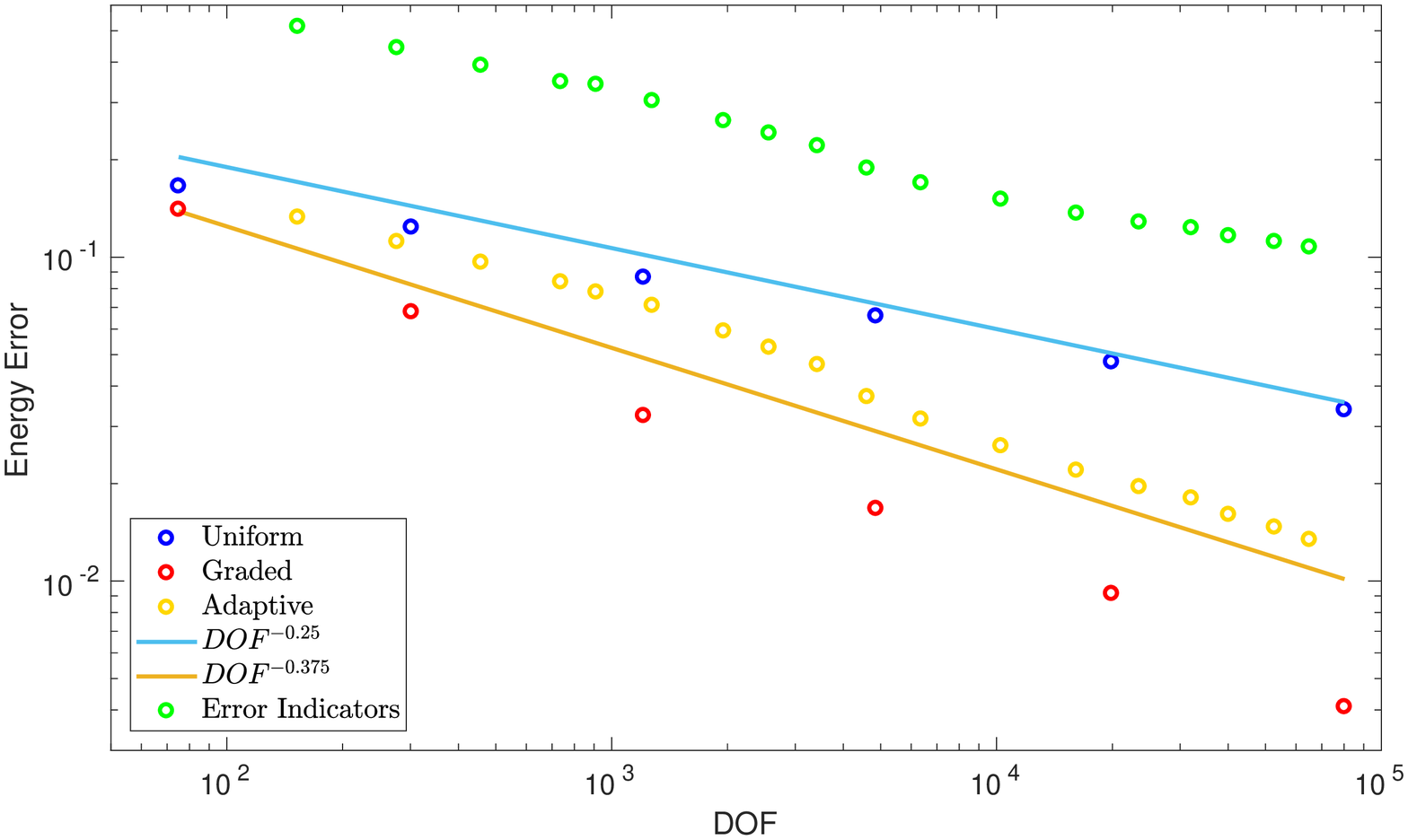}
\caption{Error in the energy norm for the variational inequality in Example~\ref{ex:1}.}
\label{f:Convergence_ineq_ell}
\end{figure}
\begin{example}\label{ex:ElFrict}
We consider the mixed formulation of the fractional friction problem \eqref{p:ell_mixed} in $\Omega = B_1(0)$, with $s = 0.6$ and $f = 1$. We compare the Galerkin solution to \eqref{p:ell_mixed_disc} by piecewise linear finite elements on uniform, graded, and adaptively refined meshes with a benchmark solution on an adaptively generated mesh with $228140$ degrees of freedom. Figure~\ref{f:ElFrict} displays the solution of the friction problem. Note that the Lagrange multiplier is discontinuous in places where $u$ changes sign. Figure~\ref{f:ElFrictConv} shows the error in the $H^s(\Omega)$ norm for different meshes in terms of the degrees of freedom. The rate of convergence in terms of degrees of freedom is $-0.220$ for uniform meshes, $-0.454$ for $2$-graded meshes, and $-0.429$ for adaptively generated meshes. This corresponds to a convergence rate of $0.440$ (uniform), $0.908$ ($2$-graded) in terms of the mesh size $h$.
\end{example}

\begin{figure}
\centering
\includegraphics[width = \textwidth]{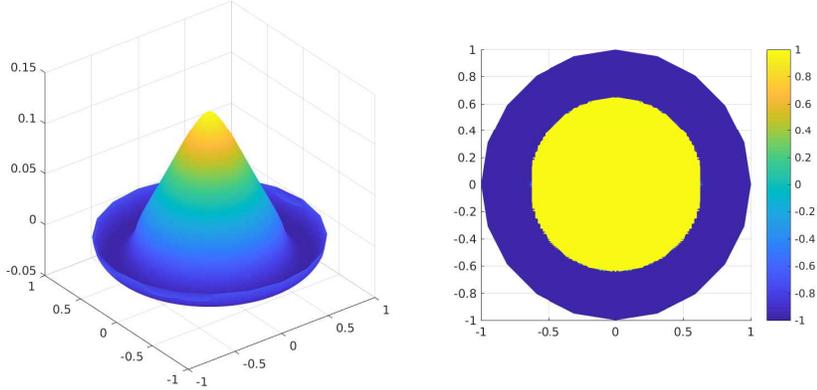}
\caption{Solution $u_h$ (left) and $\lambda_H$ (right) of Example~\ref{ex:ElFrict} on a mesh with $142719$ degrees of freedom.}
\label{f:ElFrict}
\end{figure}

\begin{figure}
\centering
\includegraphics[width = 0.7\textwidth]{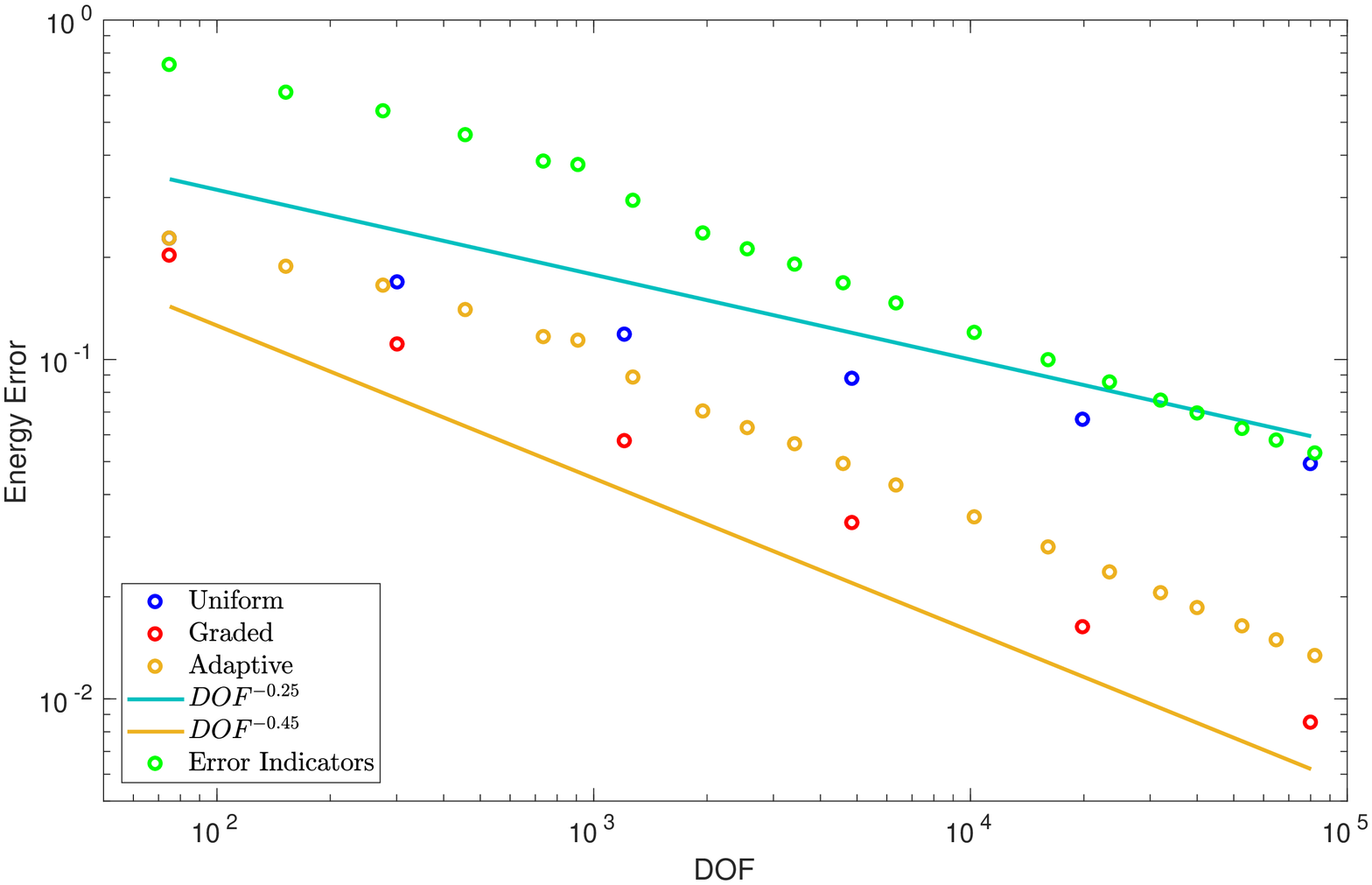}
\caption{Error in the energy norm for the variational inequality in Example~\ref{ex:ElFrict}.}
\label{f:ElFrictConv}
\end{figure}

\subsection{Dynamic contact problems}
In order to keep $\frac{\tau}{h^{2s}}$ fixed, we choose the time step $\tau \approx h^{2s}$ for uniform meshes, $\tau \approx 0.5\,h_{max}^{2s}$ for graded meshes, and local time steps $\tau \approx h^{2s}$ for adaptively generated meshes.
We first  consider a parabolic obstacle problem:
\begin{example}\label{ex:3}
We consider the mixed formulation of the fractional obstacle problem \eqref{p:para_mixed} in $\Omega = B_1(0)$, with $s = 0.5$, $f = 0$, and two different initial conditions $u_0 = 1$, $\tilde{u}_0 = 2$. We set  $T = 1$, and the obstacle $\chi$ is defined as in Example~\ref{ex:1} and depicted in Figure~\ref{f:Obstacle_chi}. We compare the Galerkin solution to \eqref{p:para_mixed_disc} by piecewise linear finite elements on uniform, graded and adaptively refined meshes with a benchmark solution on an adaptively generated mesh with $29894663$ degrees of freedom. Figure~\ref{f:ParaObstPlot} displays the solution of the obstacle problem at time $T=1$ on an adaptively refined mesh. Figure~\ref{f:para_conv} shows the error in the $H^s(\Omega)$ norm for different meshes in terms of the degrees of freedom. Again error indicators capture the slope of the error in the adaptive procedure. The rate of convergence in terms of degrees of freedom is $-0.173$ for uniform meshes, $-0.325$ for $2$-graded meshes, and $-0.319$ for adaptively generated meshes. This corresponds to a convergence rate of $0.519$ (uniform), $0.975$ ($2$-graded) in terms of the mesh size $h$. Figure~\ref{f:MeshesPara1} depicts the slices at $t = 0,\;0.4,\;0.8,\;1$ of the meshes $3,8,15$ created by the adaptive algorithm. They show strong refinement near the boundaries, as well as refinement near the contact boundary. Figure~\ref{f:ErIndVsTime} shows the error indicators of the adaptive algorithm in time for several iterations with the initial condition $\tilde{u}_0 = 2$. While for the initial condition $u_0 = 1$ the contact with the obstacle is present from time $t=0$, for initial condition $\tilde{u}_0$ the solution first touches the obstacle at time $t \approx 0.5$. The error increases rapidly at the time of first contact with the obstacle. After several iterations the adaptive algorithm equilibrates the error in space and time by refinements of space-time mesh, as shown in Figure~\ref{f:ErIndVsTime}.
\end{example}
Like for the elliptic problems, we note that the convergence closely mirrors the theoretical convergence rates \cite{Salgado}.\\

\begin{figure}
\center
\includegraphics[width = 0.7\textwidth]{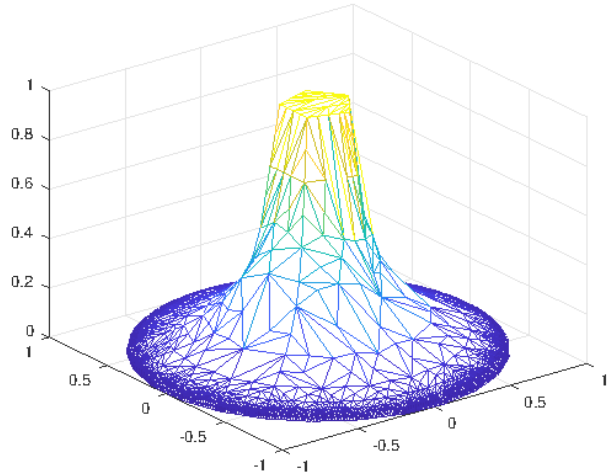}
\caption{Solution $u_{h\tau}$ of the parabolic obstacle problem from Example~\ref{ex:3} at $T = 1$.}
\label{f:ParaObstPlot}
\end{figure}
\begin{figure}[!ht]
\centering
\includegraphics[width = 0.8\textwidth]{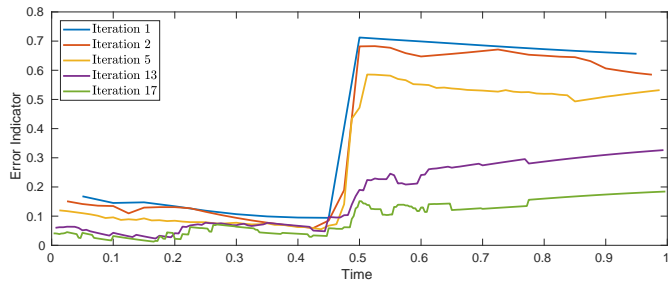}
\caption{Error indicators of adaptively refined meshes in time with the initial condition given by $\tilde{u}_0$.}
\label{f:ErIndVsTime}
\end{figure}

\begin{figure}[!ht]
\centering
\subfloat[]{\includegraphics[width = .33\textwidth]{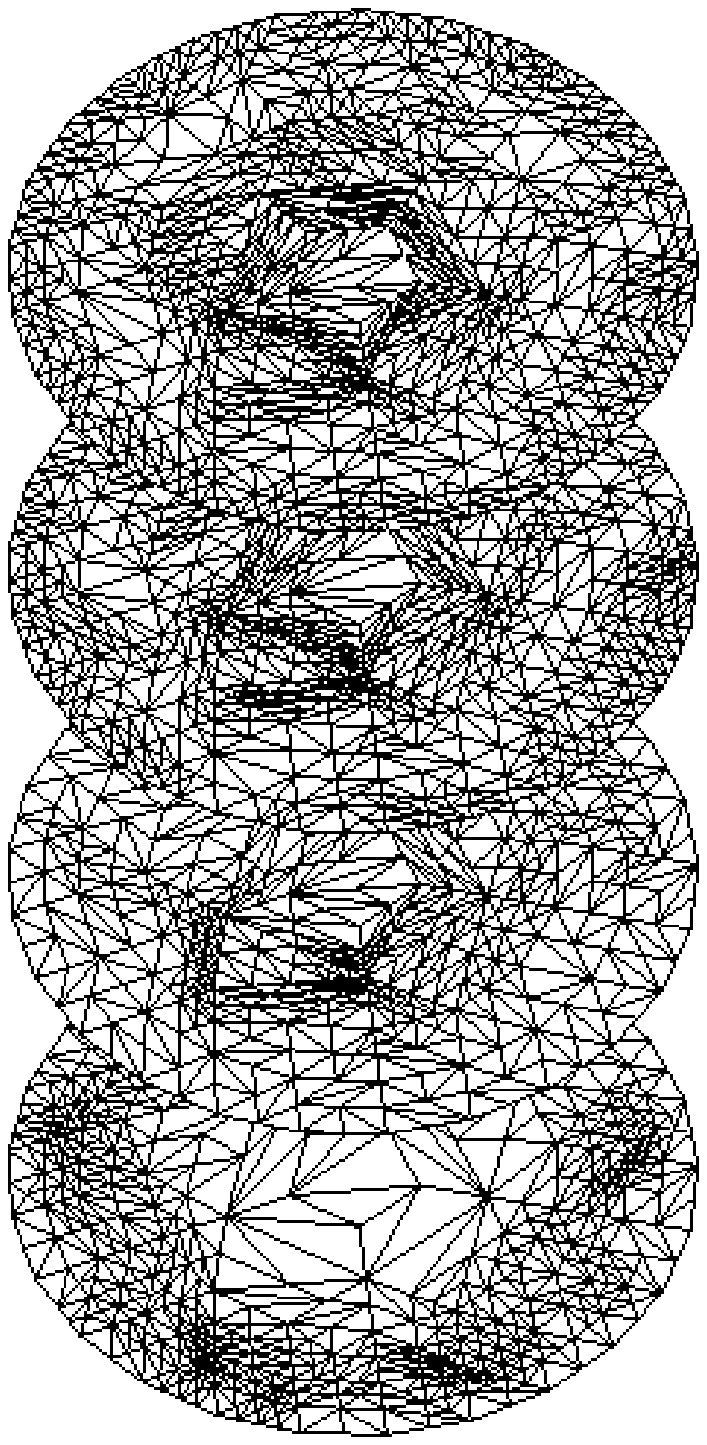}}
\subfloat[]{\includegraphics[width = .33\textwidth]{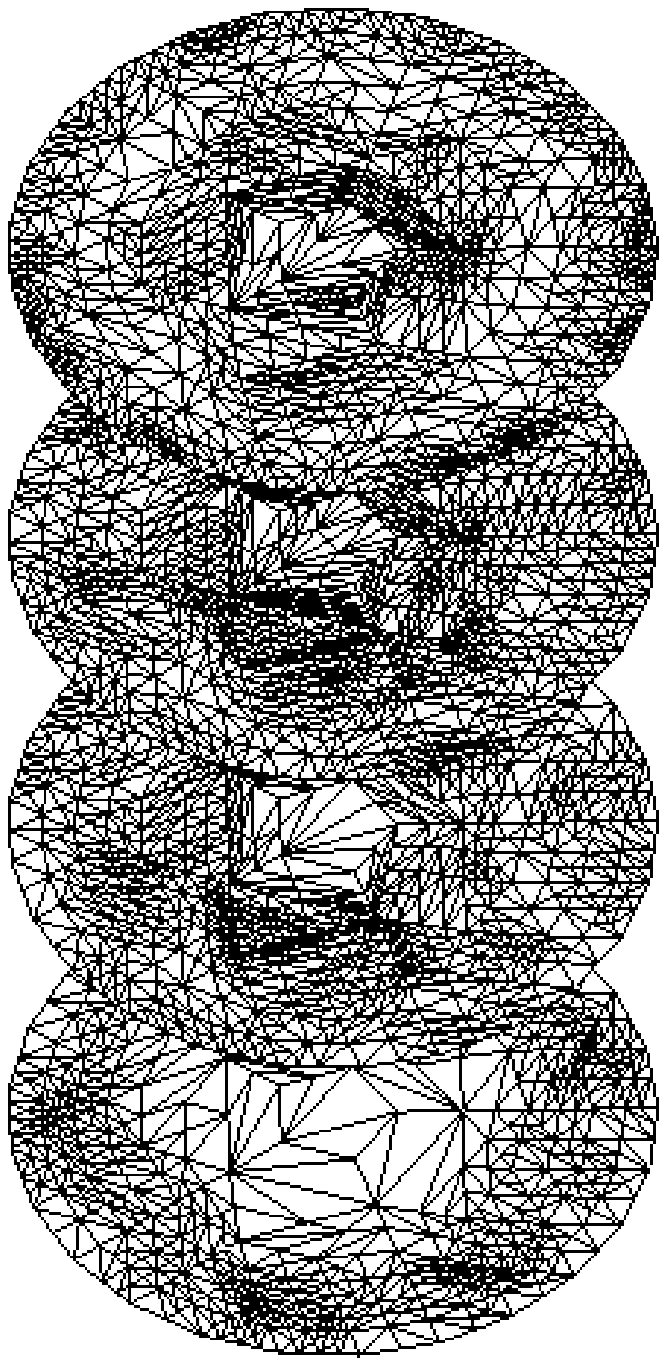}}
\subfloat[]{\includegraphics[width = .33\textwidth]{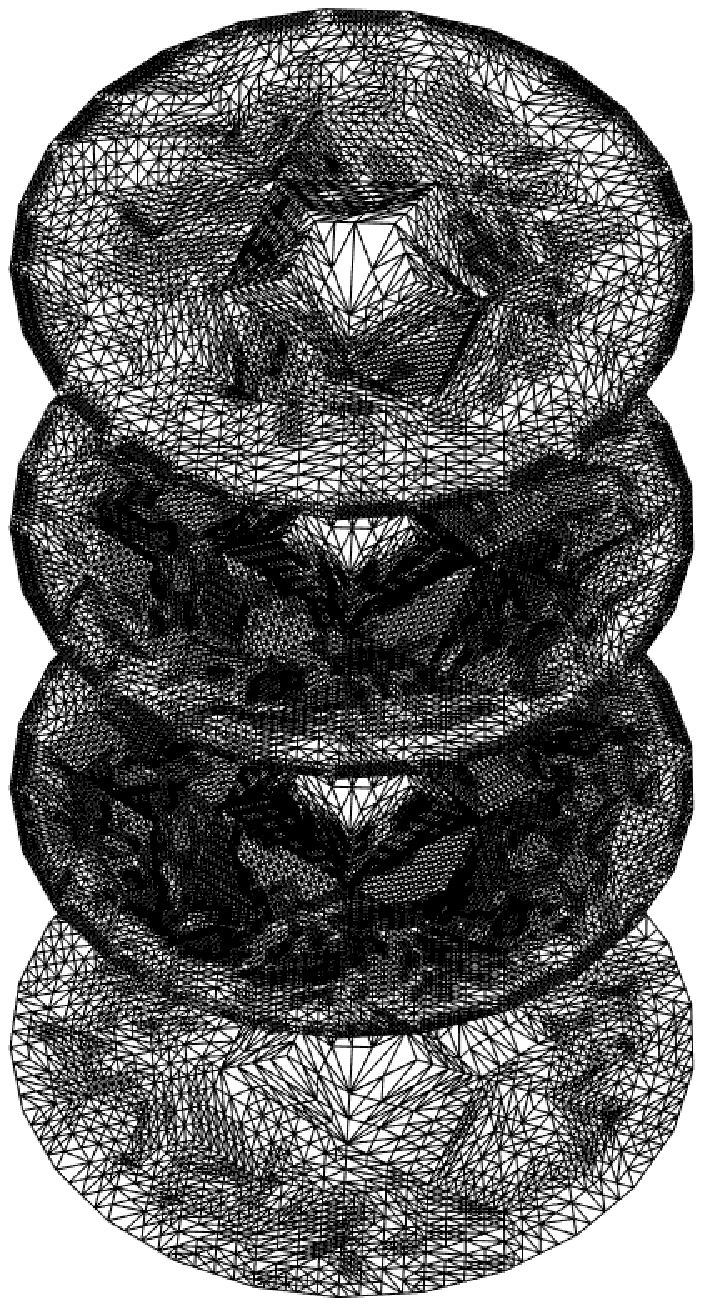}}
\caption{Adaptively refined meshes for Example~\ref{ex:3} after $3,8,15$ refinements at $t = 0,0.4,0.8,1$, from bottom to top, respectively.}
\label{f:MeshesPara1}
\end{figure}

\begin{figure}
\center
\includegraphics[width = 0.7\textwidth]{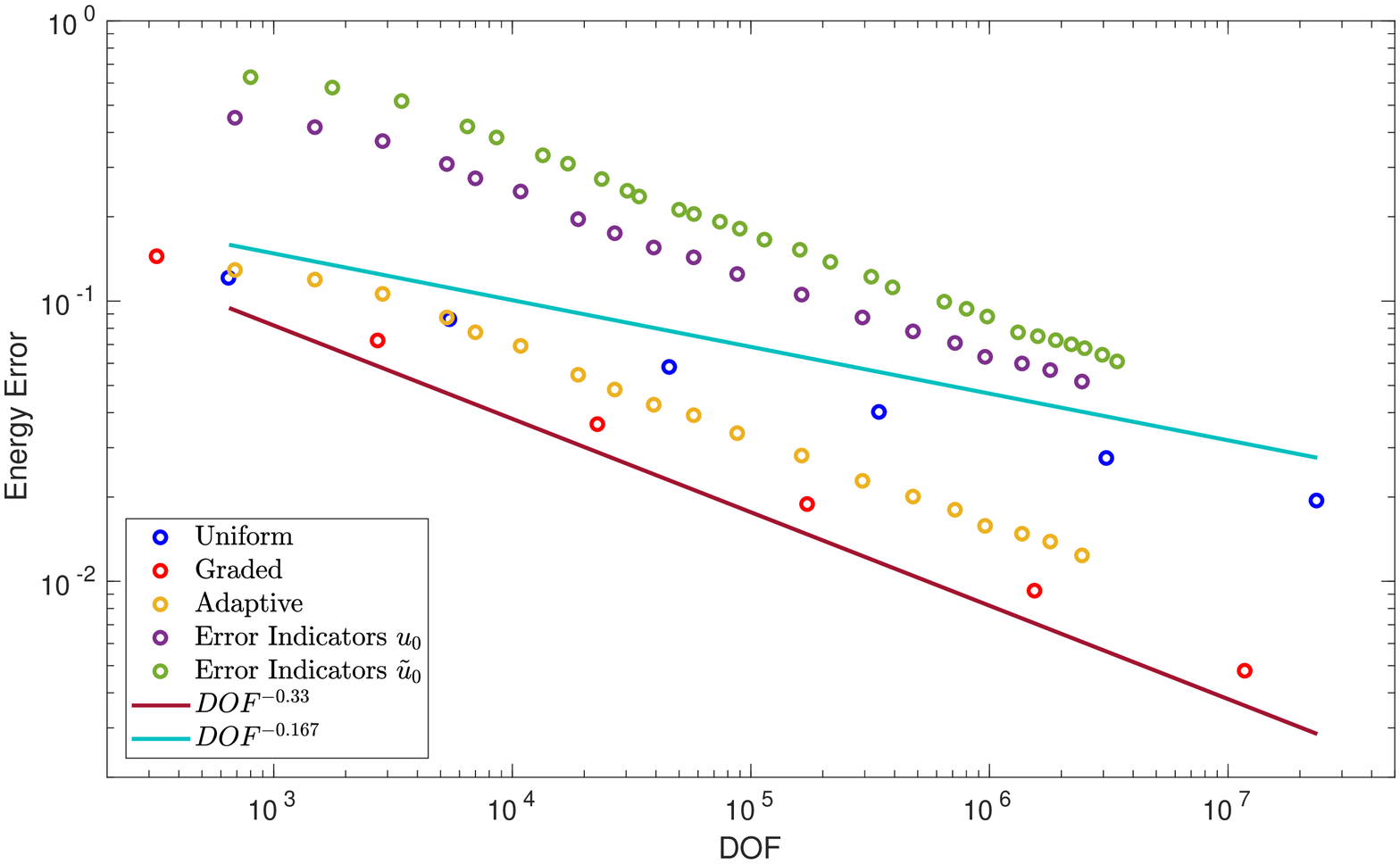}
\caption{Error in the energy norm for the variational inequality in Example~\ref{ex:3}.}
\label{f:para_conv}
\end{figure}

We finally consider a parabolic interior friction problem:
\begin{example}\label{ex:Friction}
We consider the mixed formulation of the interior fractional friction problem \eqref{p:para_mixed} in $\Omega = B_1(0)$, with $s = 0.6$, $f = 0$, $u_0 = (|x|-1)(|x|-0.6)$, $T = 1$, and $\mathcal{F} = 0.1$ in the whole domain $\Omega$. We compare the Galerkin solution to \eqref{p:para_mixed_disc} by piecewise linear finite elements on uniform, graded and adaptively refined meshes with a benchmark solution on an adaptively generated mesh with $29366872$ degrees of freedom. Figure~\ref{f:FrictionSols} shows the numerical solution of the problem at times $t = 0, 0.2, 0.4, 0.8$. Figure~\ref{f:FrictionConv} shows the error in the $H^s(\Omega)$ norm for different meshes in terms of the degrees of freedom. The error indicators capture the slope of the error in the adaptive procedure. The rate of convergence in terms of degrees of freedom is $-0.156$ for uniform meshes, $-0.322$ for $2$-graded meshes, and $-0.323$ for adaptively generated meshes. This corresponds to a convergence rate of $0.499$ (uniform), $1.030$ ($2$-graded) in terms of the mesh size $h$. The free boundary, where $\lambda$ is discontinuous, moves out of the domain as time evolves.
\end{example}

\begin{figure}[!ht]
   \centering
   \subfloat[][$t=0$]{\includegraphics[width=.45\textwidth]{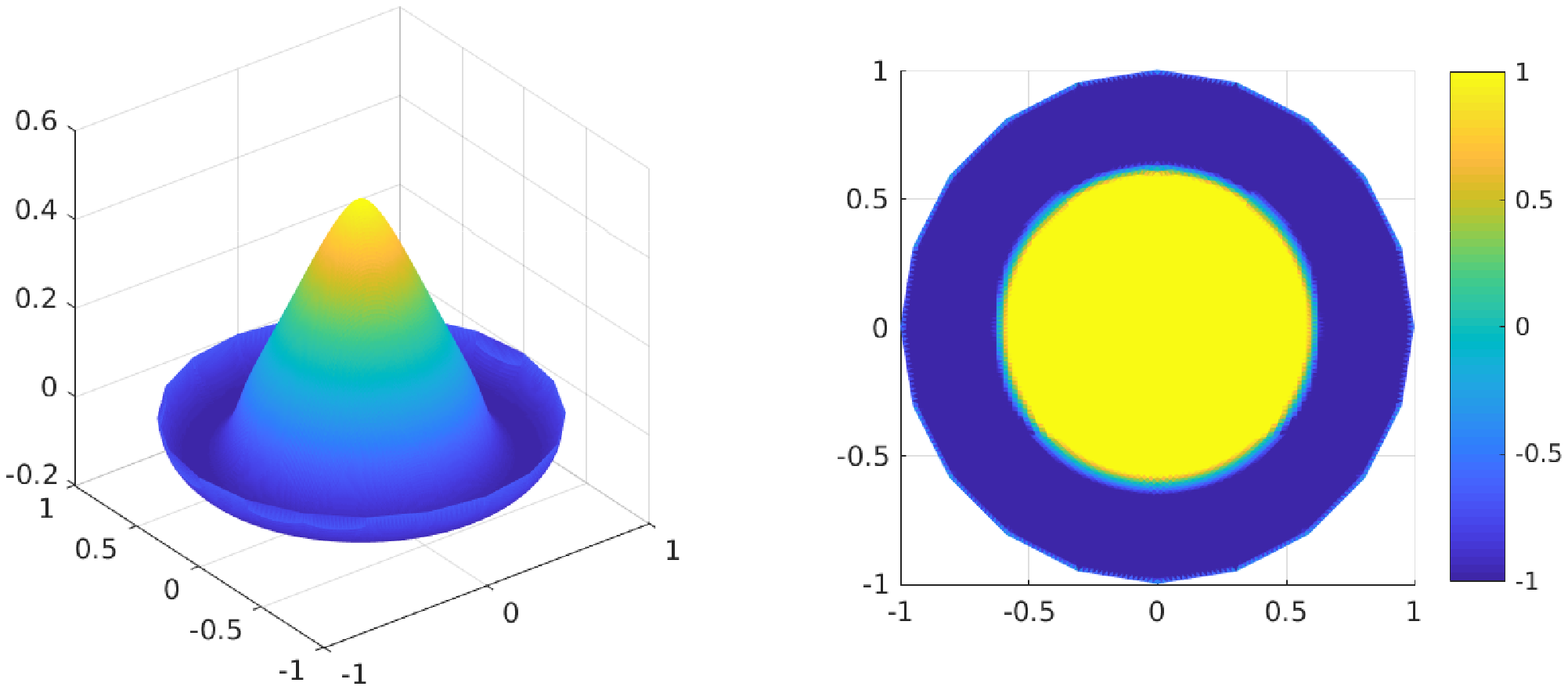}}\quad
   \subfloat[][$t=0.2$]{\includegraphics[width=.45\textwidth]{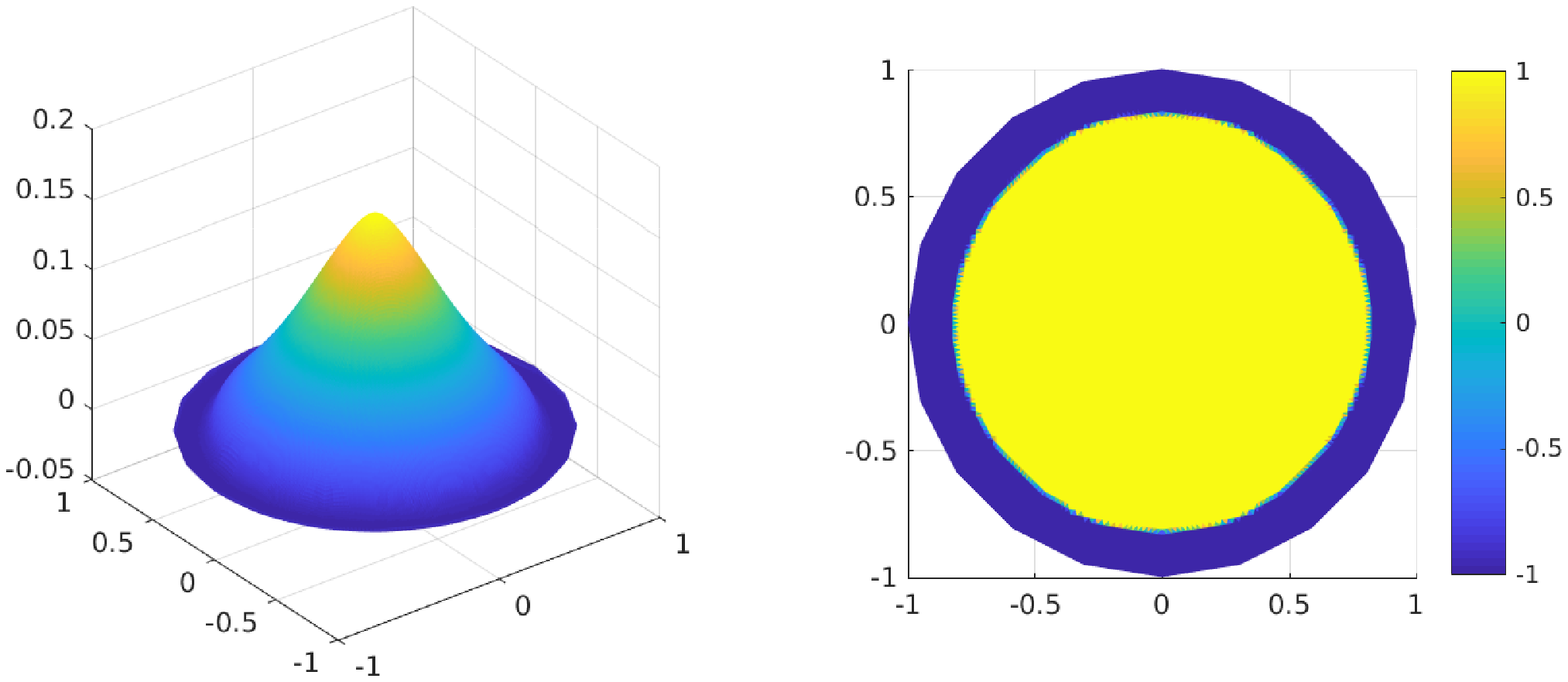}}\\
   \subfloat[][$t=0.4$]{\includegraphics[width=.45\textwidth]{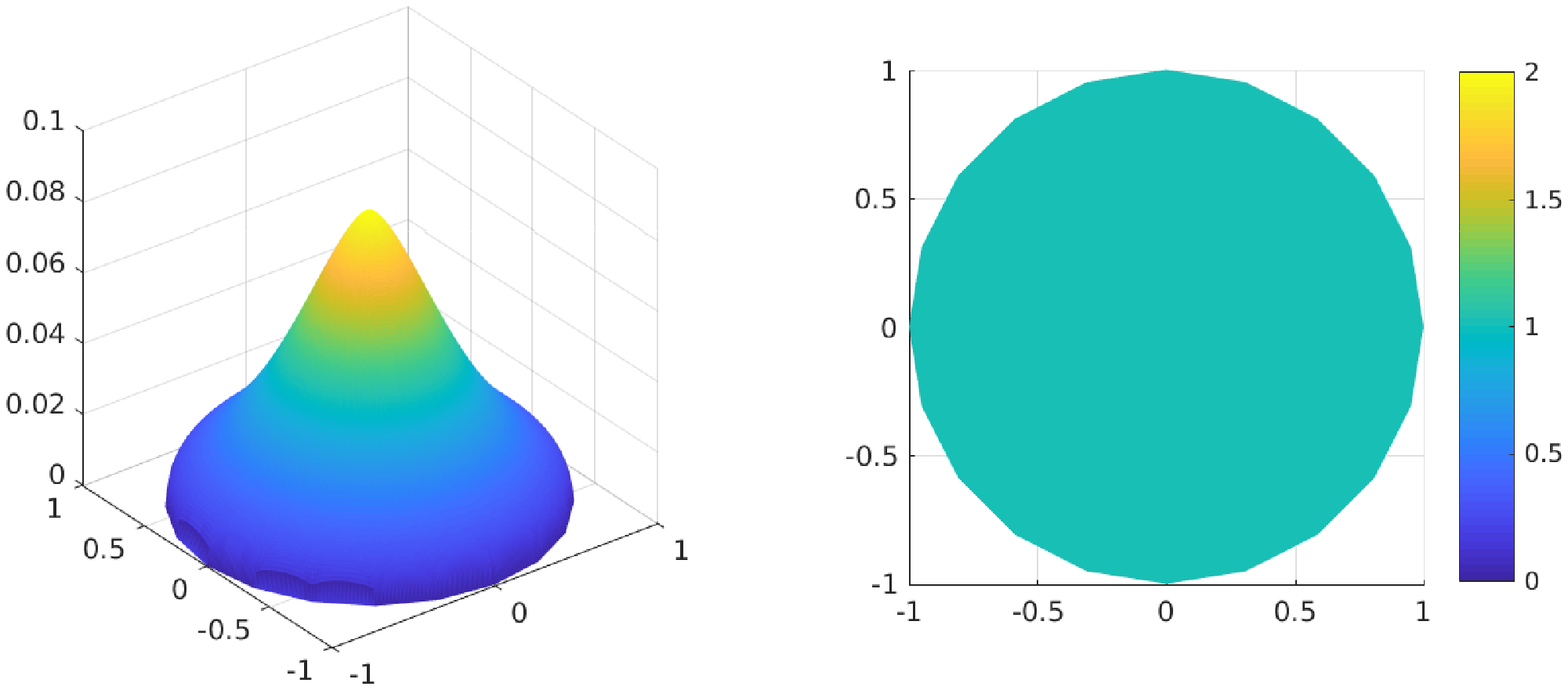}}\quad
   \subfloat[][$t=0.8$]{\includegraphics[width=.45\textwidth]{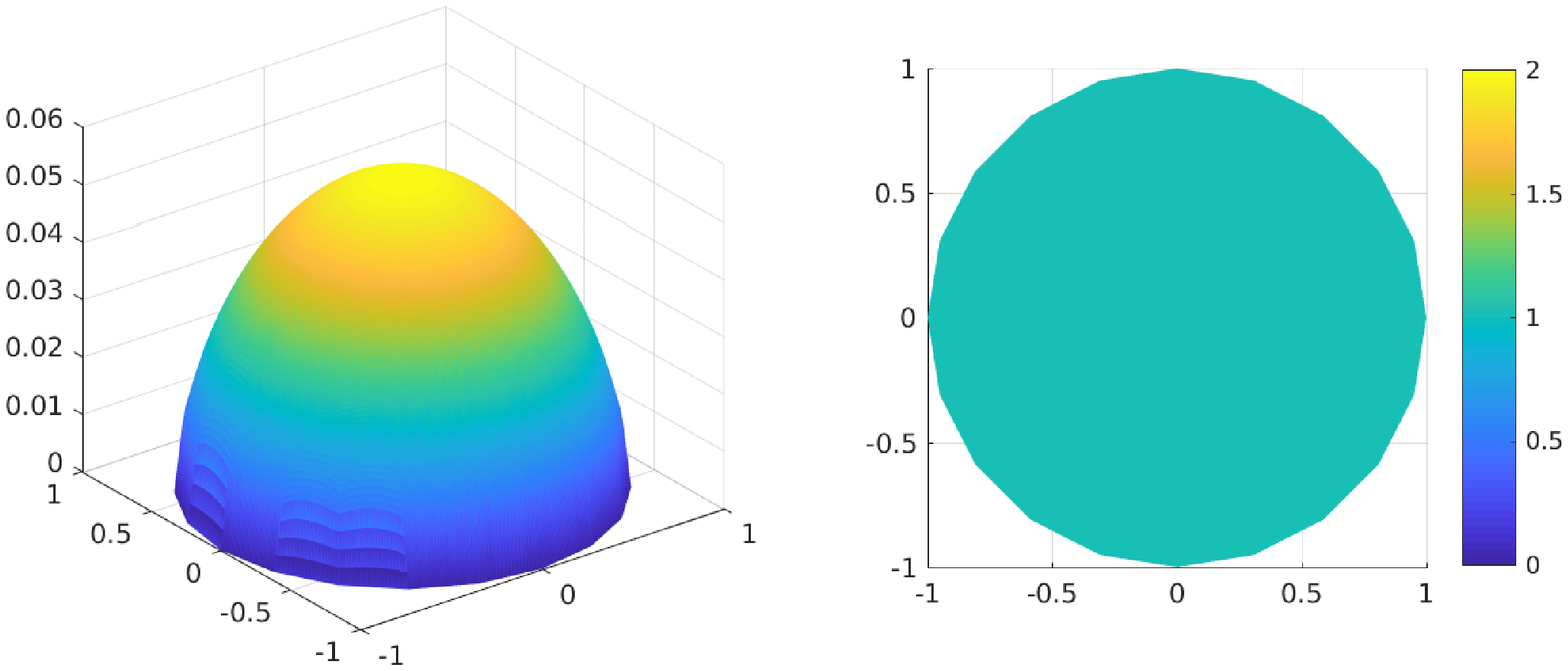}}
   \caption{Solution $(u_{h\tau},\lambda_{H\tau})$ of the interior friction problem from Example~\ref{ex:Friction} at times $t = 0, 0.2, 0.4, 0.8$ on a mesh with $139905$ degrees of freedom.}
   \label{f:FrictionSols}
\end{figure}

\begin{figure}[!ht]
\centering
\includegraphics[width = 0.7\textwidth]{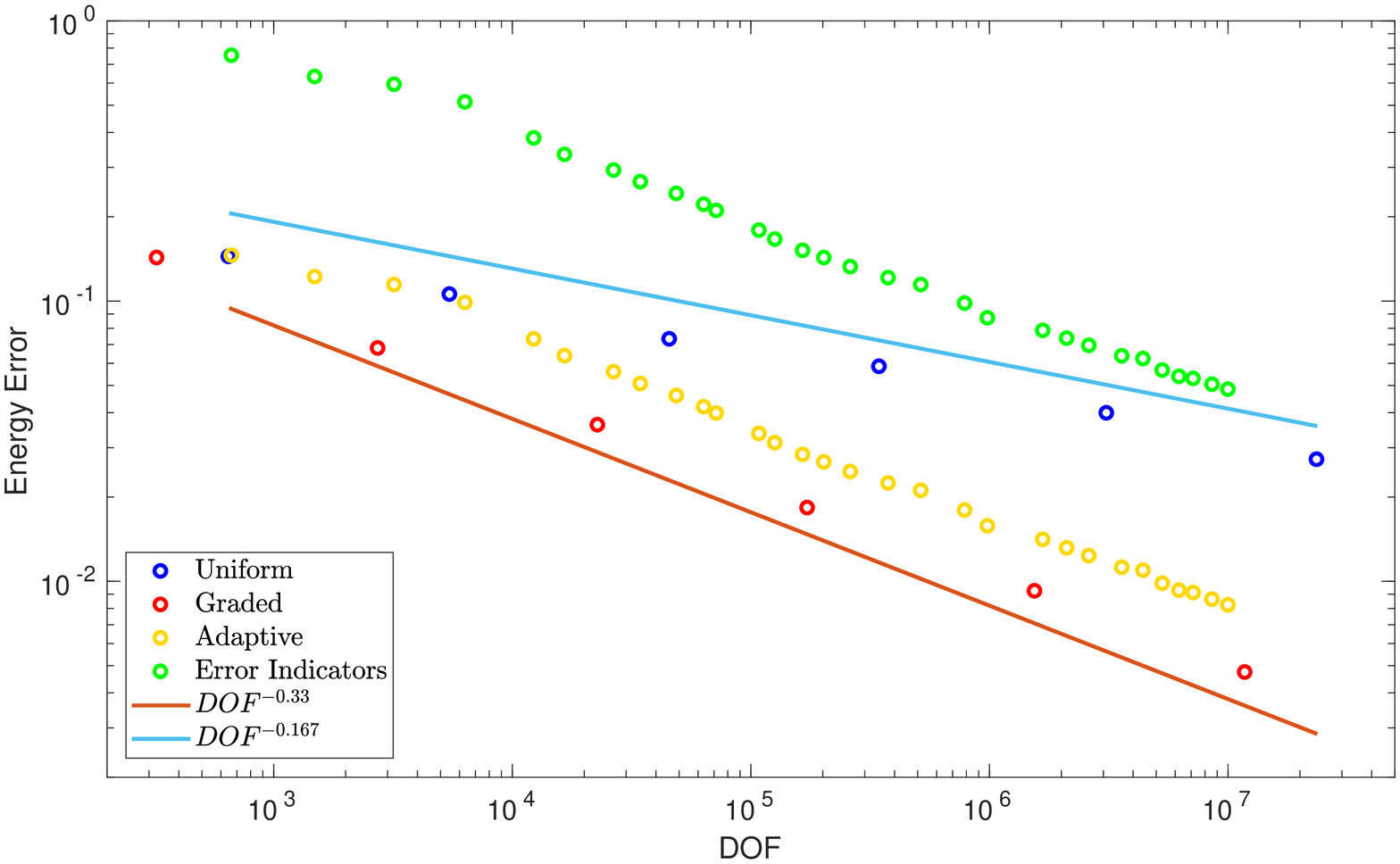}
\caption{Error in the energy norm for the friction problem from Example~\ref{ex:Friction}.}
\label{f:FrictionConv}
\end{figure}

\section{Conclusions}

Motivated by the recent interest in dynamic contact and friction problems for nonlocal differential equations in finance, image processing, mechanics and the sciences, the article provides a systematic error analysis of finite element approximations and space-time adaptive mesh refinements. The analysis of these time-dependent free boundary problems builds on ideas from time-independent boundary element methods \cite{gwinsteph}, dynamic contact and space-time adaptive techniques for parabolic problems.

Key results of this article include a priori and a posteriori error estimates and space-time adaptive numerical experiments for a mixed formulation, which directly computes the contact forces as a Lagrange multiplier. For discontinuous Galerkin methods in time, an inf-sup condition in space is sufficient for the error analysis and guarantees convergence. The analysis is complemented by corresponding results for the formulation as a variational inequality, as well as the time-independent nonlocal elliptic problem, complement the results.

Our numerical experiments illustrate the efficiency of the space-time adaptive procedure for model problems in $2$d. The adaptive method converges at the rate known for algebraically $2$-graded meshes and at twice the rate known for quasiuniform meshes. Unlike graded meshes, the adaptively generated meshes are easily applied to complex geometries and are known to resolve the space-time inhomogeneities inherent in dynamic contact. While strongly graded meshes theoretically recover optimal convergence rates for the time-independent problem, adaptive meshes are less susceptible to the floating point errors encountered for integral operators on strongly graded meshes.

Corresponding results for nonlocal elliptic problems complement the numerical experiments.

While the article analyzes the fractional Laplace operator as a model operator, the analysis extends to variational inequalities for more general nonlocal elliptic operators \cite{Nochetto}.\\

The space-time adaptive methods developed in this paper will be of interest beyond variational inequalities. They will be of use, in particular, for the  nonlinear systems of fractional diffusion equations arising in applications \cite{estrada2018, estrada2019}. 

Much recent interest has been on the stabilization and Nitsche methods for static and dynamic contact, for example \cite{BanzHeiko,burman2018nitsche,Chouly2017}. This will be addressed in future work and, in particular, will allow to circumvent the inf-sup condition on the spatial discretization.

In a different direction as mentioned in Remark~\ref{rem:63} the regularity theory and resulting optimal a priori estimates remain to be developed beyond obstacle problems.

\bibliography{VIbib}
\bibliographystyle{plain}

\end{document}